\documentclass[11pt]{article}
\usepackage[square,sort,comma,numbers]{natbib}
\usepackage[T1]{fontenc}    
\usepackage[latin1]{inputenc}

\usepackage{amsmath, lmodern}
\usepackage{amssymb}
\usepackage{amsfonts}
\usepackage{amsthm}
\usepackage{graphicx}
\usepackage{multicol}
\usepackage{enumitem, pinlabel}
\usepackage{tikz}
\usepackage{tikz-3dplot}
\usepackage{xcolor}

\usetikzlibrary{shapes,patterns,plotmarks}
\usetikzlibrary{arrows}
\usepackage{array}\usepackage{multirow}
\usepackage{graphics, mathabx}

\usepackage[babel=true]{csquotes}

  \theoremstyle{plain}
\newtheorem{theorem}{Theorem}[section]
\newtheorem{proposition}[theorem]{Proposition}
\newtheorem{notation}[theorem]{Notation}
\newtheorem{conjecture}[theorem]{Conjecture}
\newtheorem{question}[theorem]{Question}
\newtheorem{lemma}[theorem]{Lemma}
\newtheorem{problem}[theorem]{Problem}

  \theoremstyle{definition}
  \newtheorem{remark}{Remark}[section]

\AtBeginDocument{%
      \setlength\abovedisplayskip{2pt}
      \setlength\belowdisplayskip{2pt}}
\setlist{nolistsep}

\newcommand{\N}{\mathbb{N}}
\newcommand{\Z}{\mathbb{Z}}
\newcommand{\Q}{\mathbb{Q}}
\newcommand{\R}{\mathbb{R}}
\newcommand{\C}{\mathbb{C}}

\newcommand{\K}{\mathcal{K}}
\newcommand{\hz}{\hat{Z}} 
\newcommand{\Ks}{\K\setminus \Sigma}
\newcommand{\dni}{\Delta^{n+1}_i}
\newcommand{\4}{4\!\times\! 4}
\newcommand{\x}{3}
\newcommand{\y}{3}
\renewcommand{\a}{3}
\renewcommand{\b}{-2}
\newcommand{\cman}{3}
\renewcommand{\d}{6}
\newcommand{\e}{-20}
\newcommand{\V}{\mathcal{V}}
\newcommand{\W}{\mathcal{W}}

\newcommand{\hs}{\hat{\Sigma}}
\newcommand{\A}{\mathcal{A}}
\newcommand{\D}{\mathcal{D}}

\newcommand{\uZ}{\underline{Z}}
\newcommand{\Dm}[1]{{\D_m}^{\!\!\!\!#1}
}
\newcommand{\Dmt}[1]{{\tilde{\D}_m}^{\hspace*{0.5pt}#1}
}
\newcommand{\qtcat}{\raisebox{-0.25em}{\labellist
\small\hair 2pt
\pinlabel $+$ at 190 63
\endlabellist \includegraphics [scale=0.1]{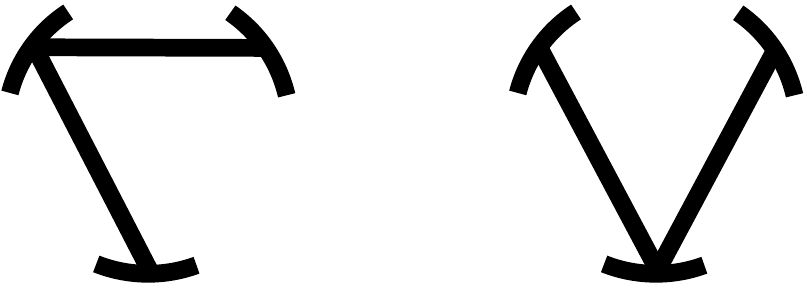}}}
\newcommand{\qtcatt}{\raisebox{-0.25em}{\labellist
\small\hair 2pt
\pinlabel $+$ at 190 63
\endlabellist \includegraphics [scale=0.1]{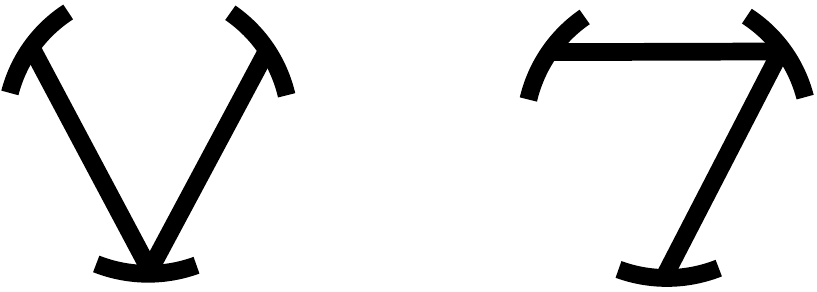}}}
\renewcommand{\Z}{\mathbb{Z}}
\renewcommand{\R}{\mathbb{R}}
\renewcommand{\S}{\mathbb{S}}

\newcommand{\Ker}{\operatorname{Ker}}

\renewcommand{\Im}{\operatorname{Im}}

 \DeclareMathSizes{10}   {10}   {7}{6}
  \DeclareMathSizes{10.95}{10.95}{8}{6}
 
\usepackage[font=small]{caption}

\title{A Kontsevich integral of order $1$}

\author{Arnaud Mortier} 

\newcommand{\cgraph}{\raisebox{-0.75em}{\labellist
\small\hair 2pt
\pinlabel $1$ at 5 90
\pinlabel $2$ at 5 5
\pinlabel $3$ at 95 5
\pinlabel $4$ at 95 90
\endlabellist \includegraphics[height=2.5em]{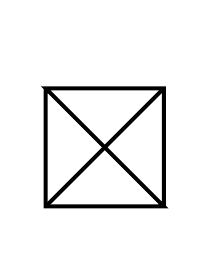}}}

\newcommand{\lmbn}{\raisebox{-0.75em}{\labellist
\small\hair 2pt
\pinlabel $+$ at 170 42
\pinlabel $i$ at 30 -19
\pinlabel $j$ at 73 -19
\pinlabel $k$ at 116 -16
\pinlabel $i$ at 230 -19
\pinlabel $j$ at 273 -19
\pinlabel $k$ at 316 -16
\endlabellist \includegraphics [height=2.7em]{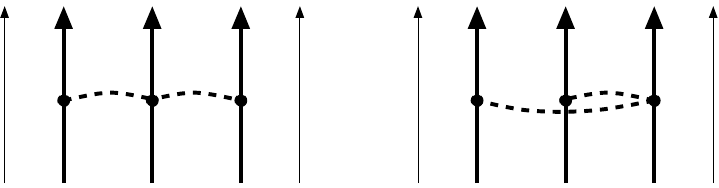}}}

\newcommand{\lmbnn}{\raisebox{-0.75em}{\labellist
\small\hair 2pt
\pinlabel $-$ at 170 42
\pinlabel $i$ at 30 -19
\pinlabel $j$ at 73 -19
\pinlabel $k$ at 116 -16
\pinlabel $i$ at 230 -19
\pinlabel $j$ at 273 -19
\pinlabel $k$ at 316 -16
\endlabellist \includegraphics [height=2.7em]{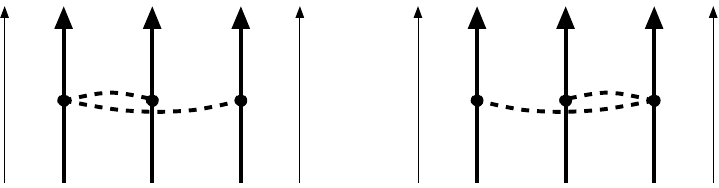}}}
\newcommand{\cgraphbis}{\raisebox{-0.75em}{\labellist
\small\hair 2pt
\pinlabel $1$ at 9 92
\pinlabel $2$ at 49 -3
\pinlabel $3$ at 93 92
\endlabellist \includegraphics[height=2.5em]{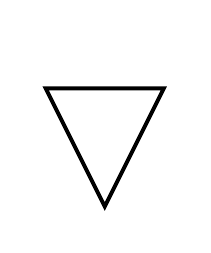}}}
\newcommand{\cgraphter}{\raisebox{-0.75em}{\labellist
\small\hair 2pt
\pinlabel $1$ at 9 92
\pinlabel $2$ at 49 -3
\pinlabel $3$ at 93 92
\pinlabel $4$ at 139 92
\pinlabel $5$ at 179 -3
\pinlabel $6$ at 223 92
\endlabellist \includegraphics[height=2.5em]{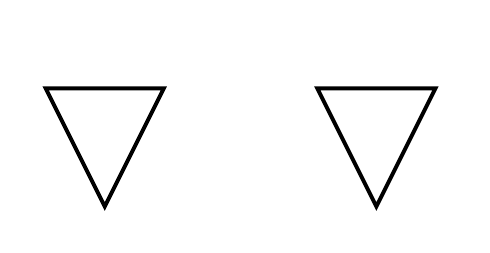}}}

\newcommand{\fTone}{\raisebox{-0.5em}{\labellist
\small\hair 2pt
\pinlabel $+$ at 72 31
\endlabellist \includegraphics[height=1.5em]{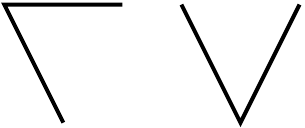}}}

\newcommand{\fTtwo}{\raisebox{-0.5em}{\labellist
\small\hair 2pt
\pinlabel $+$ at 72 31
\endlabellist \includegraphics[height=1.5em]{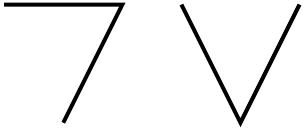}}}
\newcommand{\fTexample}{\raisebox{-0.5em}{\labellist
\small\hair 2pt
\pinlabel $+$ at 363 29
\pinlabel $+$ at 165 29
\pinlabel $+$ at 562 29
\endlabellist \includegraphics[height=1.5em]{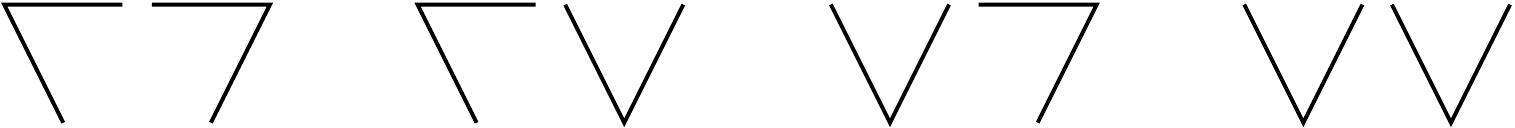}}}

\newcommand{\Span}{\operatorname{Span}}
\newcommand{\Id}{\operatorname{id}}
\newcommand{\pr}{\operatorname{pr}}
\newcommand{\rot}{\operatorname{rot}}

\newcommand{\lk}{\operatorname{lk}}

\newskip\stdskip                      
\begin{document}
\date{}
\maketitle
\begin{abstract}    
\footnotesize
We define a $1$-cocycle in the space of long knots that is a natural generalization of the Kontsevich integral seen as a $0$-cocycle. It involves a $2$-form that generalizes the Knizhnik--Zamolodchikov connection. We show that the well-known close relationship between the Kontsevich integral and Vassiliev invariants 
(via the algebra of chord diagrams and $1$T-$4$T relations) 
is preserved between our integral and Vassiliev $1$-cocycles, via a change of variable similar to the one that led Birman--Lin to discover the $4$T relations. We explain how this construction is related to Cirio--Faria Martins' categorification of the Knizhnik--Zamolodchikov connection.
\end{abstract}

\begin{center}\footnotesize 
\textit{To Joan S Birman and Xiao-Song Lin}
\end{center}

\tableofcontents
\bigskip

The Kontsevich integral is a knot invariant that simultaneously realizes all of the Vassiliev invariants \cite{Kontsevich, BarNatanVKI,VassilievBook}. The relationship between the two frameworks is a very intimate bond that is spectacularly revealed in light of an algebra of chord diagrams that plays a central role in both constructions. We owe the proof that the Vassiliev settings can be reduced to the defining relations of that algebra to Birman--Lin \cite{BirmanLin}. Now despite the massive momentum brought forward to try and understand these invariants, it seems to have been left aside that Vassiliev invariants are at level $0$ of a whole cohomology theory of the topological moduli space of knots. Indeed, since Vassiliev~\cite{VassilievBook} and Hatcher~\cite{HatcherTopologicalmoduli}  (followed by Budney~\cite{BudneyHomotopyType} and Budney--Cohen~\cite{BudneyHomology}) laid the groundwork for the study of the topology of the space of knots, relatively few attempts have been made to build actual realizations of $1$-cocycles or higher invariants: see \cite{VassilievTT2},  
\cite{Sakai1, Sakai2},  \cite{MortierCADS, FT1cocycles}, \cite{FiedlerQuantum1cocycles, FiedlerSingularization, FiedlerPolyCocy}, \cite{Longoni, PelattSinha}.

In light of the bond between the Kontsevich theory and Vassiliev knot invariants, it is only natural to expect that the Kontsevich integral will naturally generalize to higher dimensional invariants. The purpose of this article is to build a $1$-cocycle from that perspective. 

There was one earlier attempt to build $1$-cocycles by means of integrals: Sakai \cite{Sakai1, Sakai2} uses configuration space integrals building on methods developed by Bott--Taubes \cite{BottTaubes}, and further by Cattaneo--Cotta-Ramusino--Longoni \cite{CattaneoCottaLongoni} to construct Vassiliev cohomology classes for knots in higher dimensional Euclidean spaces. Topologically, Sakai's approach  uses a generalization of the linking number of a two-component link (where every pair of points contributes) rather than a two-component braid (where only pairs of points at the same altitude contribute, as in the present work).

The present construction can be seen as a realization of Cirio and Faria Martins' categorification of the Knizhnik--Zamolodchikov connection \cite{CirioFaria}. Part of a dictionary is presented briefly at the end of Appendix~\ref{Appen:1}. This aspect shall be developed in a forthcoming paper joint with Faria Martins.

The paper is organized as follows:

Section~\ref{sec:0} contains an account of Vassiliev's cohomology theory for long knots, followed by the settings of the Kontsevich integral. Several ways to think of the crucial four-term ($4$T) relations are given along the way.

In Section~\ref{sec:A}, we set up the target space $\A^1$ of our $1$-cocycle, which is spanned by some degenerate version of ordinary chord diagrams, subject to a number of relations that can be interpreted as \enquote{higher $4$T relations}.

In Section~\ref{sec:leaf} is defined a prototype version of our $1$-cocycle in the space of Morse knots---that is, essentially, knots with a given finite number of \enquote{caps and cups} with respect to an altitude function. The situation is reminiscent of how the Kontsevich integral is naturally defined first on Morse knots before one can deal with \enquote{Reidemeister I moves}. We end the section by investigating the first elementary properties of the cocycle, which will be essential in the next section.

In Section~\ref{sec:correctZ1}, we extend the definition of our cocycle to the whole space of long knots, using essentially the same correction term as the Kontsevich integral.
We then investigate further its functorial properties, showing along the way that our target space $\A^1$ is \enquote{almost} a bimodule over the target space of the Kontsevich integral.

In Section~\ref{sec:Vass}, we introduce the notion of a weight system of order $1$ and define the change of variable that will turn the set of relations defining Vassiliev's $1$-cocycles into the exact same set of relations that defines $\A^1$. The situation is analogous to how the $4$T relations show up in both Vassiliev and Kontsevich's frameworks. Although we cannot formally prove that our invariant is a Vassiliev $1$-cocycle, we show in Subsection~\ref{sec:rot} that its degree $3$ part can be explicitly evaluated on a cycle that is canonically defined in each path-component of the space of knots (that is, for each knot type), and we identify the resulting knot invariant as the first non-trivial Vassiliev invariant---the coefficient of \,\raisebox{-.2cm}{\includegraphics[scale= 0.3]{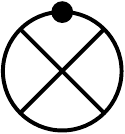}} in the Kontsevich integral, known as the Casson invariant.

Section~\ref{sec:open} lists some open questions raised by our construction. 

\subsection*{Acknowledgements}
I wish to thank the referees for their very careful reading and suggestions,  which substantially improved the readability. I am grateful to Joan Birman for her moral support and suggestions over the last year of writing this paper; to Victoria Lebed for her support and many fruitful discussions; to Jo\~ao Faria Martins for extremely valuable comments the scope of which go beyond this paper; and to Victor A. Vassiliev for his confirmation of a typo in \cite{VassilievCalc}. 

\section{Some background}\label{sec:0}

\subsection{The space of knots and Vassiliev's cohomology}\label{sec:VC}


A \textit{long knot} is an embedding $\R\to \R^3$ that coincides with $t\mapsto (0,0,t)$ outside a compact subset of $\R$. Most research in knot theory ignores the parametrization and focuses on isotopy classes of knots---which often brings down the research to combinatorial issues. It may not be obvious at first glance what added value can arise from dealing with parametrizations. First, it allows one to see the set of all knots as a topological space, and consider higher order invariants (where usual isotopy invariants of knots are $0$-cohomology classes). Second, thanks to a number of structural observations added to technical brilliancy, Vassiliev \cite{Vassiliev1990} was able to define a degree of complexity on such cohomology classes, as well as a systematic method to find those of finite degree. Let us outline these ideas.

Following Vassiliev, let $\K$ denote the topological space made of all smooth maps $\R\to \R^3$ that coincide with $t\mapsto (0,0,t)$ outside a compact subset of $\R$, endowed with the topology of uniform convergence. The subset of elements of $\K$ which are not embeddings is the \textit{discriminant}, denoted by $\Sigma$. The goal is to understand $\K\setminus \Sigma$, the space of knots.

The first observation is that $\K$ is a contractible space (it deformation retracts on the unknot), so that Alexander duality applies: up to some subtleties\footnote{Constant $0$-cocycles are not linking forms, as a linking form vanishes at infinity---that is the only exceptional case. Also, $\K$ is infinite-dimensional, so the whole construction requires finite-dimensional approximations together with a stabilization theorem.}, any cohomology class in $\Ks$ arises as the linking form with some cycle in $\Sigma$, and there is a natural isomorphism between the corresponding (co)homology groups.

Second, observe that elements of $\Sigma$ are singular maps, and such maps can be sorted according to how singular they are---via the codimension of the set of maps with a similar pattern of singularities. This makes $\Sigma$ into a stratified space; for instance, the top-dimensional stratum is the set of all maps with exactly one double point ($K(a)=K(b)$ with $a\neq b$) and a nowhere vanishing derivative. Going through such a stratum is commonly described as performing a \textit{crossing change} on a knot. 

The degree of complexity of a cohomology class is defined via dual cycles in $\Sigma$, by investigating how these cycles interact with strata of any dimension. Vassiliev's method uses a topological blow-up: expand every stratum by a cartesian product with an artificial simplex whose dimension increases with the stratum's codimension, so that there is room for a chain of any dimension to live in any stratum. The resulting space $\hs$ is called the \textit{resolved discriminant}. Now a chain that is a cycle in $\Sigma$ may not be a cycle in $\hs$: pieces that used to fit together may now lie far apart due to the blow-up---see Figure~\ref{fig:resolution}. If one can complete the chain back into a cycle in $\hs$, using only cells of depth\footnote{The depth index will not be detailed here, but it can be roughly understood by Figure~\ref{fig:resolution}: the cell labelled $e$ lies one level deeper than the other cells on the picture.} less than some integer $n$, then the original cycle and its dual cohomology class are said to have \textit{complexity less than} $n$. In this case, the deepest piece of the cycle is called a \textit{principal part} of the cohomology class.

In other words, (co)homology classes of finite complexity are those that arise from the spectral sequence induced by the filtration of $\hs$ by the depth index, and a principal part is a relative cycle from the first sheet of the sequence.
 
\begin{figure}[ht!]
\tdplotsetmaincoords{70}{300}
\begin{center}
\begin{tikzpicture}[tdplot_main_coords,font=\sffamily, scale=0.4]
  \draw[fill=black, opacity = 0.3] (\x-\d,\x-\e,0)  -- (\x-\d,-\x-\e,0)  -- (-\x-\d,-\x-\e,0)  -- (-\x-\d,\x-\e,0) -- cycle;
  \draw[fill=black, opacity = 0.3, pattern = horizontal lines] (0-\d,\x-\e,0)  -- (0-\d,-\x-\e,0)  -- (-\x-\d,-\x-\e,0)  -- (-\x-\d,\x-\e,0) -- cycle;
  \draw[fill=black, opacity = 0.3] (0-\d,\x-\e,\x)  -- (0-\d,\x-\e,-\x)  -- (0-\d,-\x-\e,-\x)  -- (0-\d,-\x-\e,\x) -- cycle;
\draw[fill=black, opacity = 0.3, pattern = horizontal lines] (0-\d,\x-\e,\x)  -- (0-\d,\x-\e,0)  -- (0-\d,-\x-\e,0)  -- (0-\d,-\x-\e,\x) -- cycle;
   \draw (0-\d,\x-\e,0)  -- (0-\d,-\x-\e,0);

\draw[fill=black, opacity = 0.3] (\y-\a,\y-\b,0-\cman)  -- (\y-\a,-\y-\b,0-\cman)  -- (-\y-\a,-\y-\b,0-\cman)  -- (-\y-\a,\y-\b,0-\cman) -- cycle;
  \draw[fill=black, opacity = 0.3, pattern = horizontal lines] (0-\a,\y-\b,0-\cman)  -- (0-\a,-\y-\b,0-\cman)  -- (-\y-\a,-\y-\b,0-\cman)  -- (-\y-\a,\y-\b,0-\cman) -- cycle;
  \draw[fill=black, opacity = 0.3] (0+\a,\y+\b,\y+\cman)  -- (0+\a,\y+\b,-\y+\cman)  -- (0+\a,-\y+\b,-\y+\cman)  -- (0+\a,-\y+\b,\y+\cman) -- cycle;
\draw[fill=black, opacity = 0.3, pattern = horizontal lines] (0+\a,\y+\b,\y+\cman)  -- (0+\a,\y+\b,0+\cman)  -- (0+\a,-\y+\b,0+\cman)  -- (0+\a,-\y+\b,\y+\cman) -- cycle;
   \draw (0+\a,\y+\b,0+\cman)  -- (0+\a,-\y+\b,0+\cman);
   \draw (0-\a,\y-\b,0-\cman)  -- (0-\a,-\y-\b,0-\cman);

\draw[fill=black, opacity = 0.2] (0-\a,\y-\b,0-\cman)  -- (0-\a,-\y-\b,0-\cman) -- (0+\a,-\y+\b,0+\cman) -- (0+\a,\y+\b,0+\cman) -- cycle;

\node[] at (-15,15) {$d$};
\node[] at (-8.5,23.7) {$a$};
\node[] at (-4,16) {$b$};
\node[] at (2.6,25) {$c$};

\node[] at (1,4) {$e$};

\node[] at (18.6,7) {$c$};
\node[] at (1,-5) {$d$};
\node[] at (-12.5,1.6) {$a$};
\node[] at (-8,-6.1) {$b$};

\end{tikzpicture}
 
\end{center}
\caption{On the left, top-dimensional strata in $\Sigma$ ($1$-codimensional in $\K$), corresponding to some sets of maps with one double point, meet \enquote{transversely} at a set of maps with two double points. Going from area $a$ to area $b$ (or $c$ to $d$) amounts to performing a crossing change on a singular knot, away from its double point. Part of a cycle in $\Sigma$ is depicted as a hatched area (where the weights $a$ and $c$ are equal, and $b=d=0$).\newline
On the right, in the resolved discriminant, the chain that was a cycle in $\Sigma$ now has boundary. To try and make it a cycle again, one needs to acknowledge the use of higher codimensional strata---think of a toll bridge---here with a local weight $e=a-b=c-d$.}\label{fig:resolution}
\end{figure}

\subsubsection{Example: usual knot invariants, and the four-term relation seen from the outside}\label{sec:exuki}
Let $v$ be a knot invariant with values in $\Q$ (or any abelian group)---in other words $v\in H^0(\K\setminus \Sigma,\Q)$. The Alexander dual of $v$ is a cycle $\sigma$ of maximal dimension in $\Sigma$, with which $v$ is the linking form (up to a constant invariant). The local weights that define $\sigma$ can therefore be computed as the value by which $v$ jumps when crossing a top-dimensional stratum of $\Sigma$. This is the meaning of the \textit{derivative formula}: $$v'\left({\raisebox{-0.4em}{\includegraphics[scale=0.12]{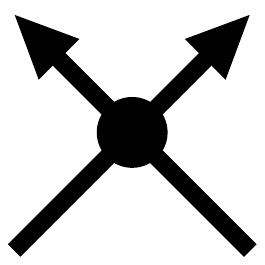}}}\right)=v\left({\raisebox{-0.4em}{\includegraphics[scale=0.12]{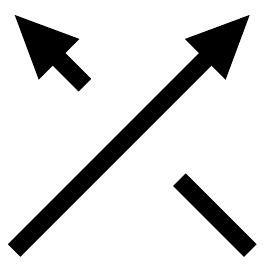}}}\right)-v\left({\raisebox{-0.4em}{\includegraphics[scale=0.12]{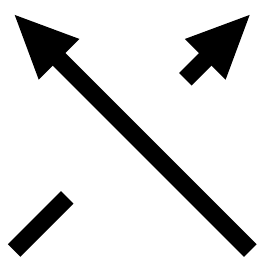}}}\right)$$
where the sign convention for positive and negative crossings is to be thought of as a co-orientation of the stratum in between---note that this convention depends only on the choice of an orientation of $\R^3$ and does not require one to consider a projection $\R^3\to \R^2$.

Now to see whether we have an invariant of finite complexity, we compute the weights given to higher and higher codimensional strata as displayed in Figure~\ref{fig:resolution}, where the equality $e=a-b=c-d$ can be written as $$v^{(2)}\left({\raisebox{-0.4em}{\includegraphics[scale=0.12]{double}\includegraphics[scale=0.12]{double}}}\right)=v'\left({\raisebox{-0.4em}{\includegraphics[scale=0.12]{double}\includegraphics[scale=0.12]{poscross}}}\right)-v'\left({\raisebox{-0.4em}{\includegraphics[scale=0.12]{double}\includegraphics[scale=0.12]{negcross}}}\right)=v'\left({\raisebox{-0.4em}{\includegraphics[scale=0.12]{poscross}\includegraphics[scale=0.12]{double}}}\right)-v'\left({\raisebox{-0.4em}{\includegraphics[scale=0.12]{negcross}\includegraphics[scale=0.12]{double}}}\right).$$
The derivative formula is thus straightforwardly extended to singular knots with arbitrarily many double points. 

%

For an invariant to be of finite complexity, one of its higher derivatives must vanish. In this case, the weights of the strata of knots with $n$ double points (where $v^{(n+1)}=0$) are invariant under crossing change, and therefore only depend on the relative order in which double points are met along the knot. A collection of such strata connected to each other via crossing changes is represented by an \textit{ordinary chord diagram}: a real line (or a circle with a dot at infinity) enhanced with finitely many disjoint \textit{chords} (pairs of distinct points), altogether regarded up to the action of $\operatorname{Homeo}_+(\R)$. The purpose of a chord here is to point at the two preimages of a double point---see Figure~\ref{fig:exCD}. The last non-zero derivative of a $0$-cocycle can therefore be entirely described by a formal linear combination of such diagrams, called the \textit{weight system} of the cocycle. 

\begin{figure}[h!]
\centering
\includegraphics[scale=0.336]{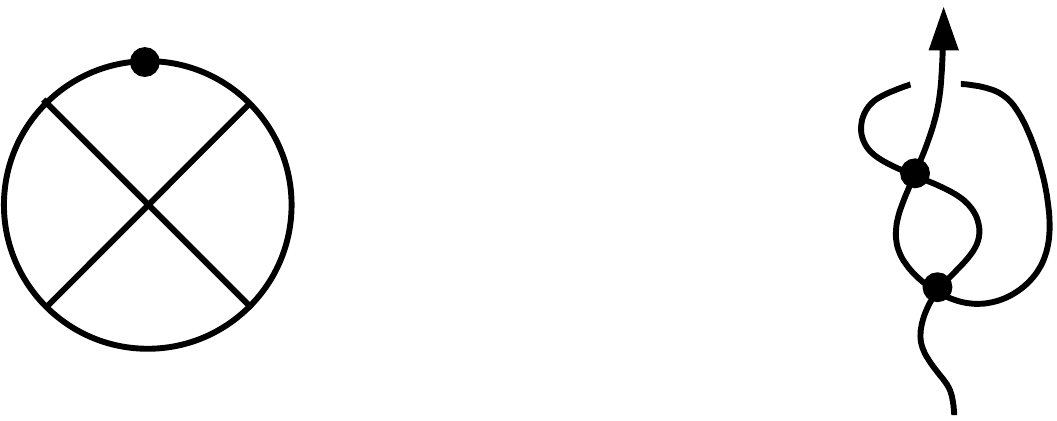}
\caption{On the left, a chord diagram that represents the codimension $2$ stratum of all singular knots with exactly two double points, met in the order $1,2,1,2$ along the knot, and a nowhere vanishing derivative. On the right is an example of such a singular knot. This chord diagram on its own happens to be the weight system of a cocycle known as the Casson invariant---the coefficient of the quadratic term in the Conway polynomial.}\label{fig:exCD}
\end{figure}
We now investigate two properties which characterize weight systems, as it will turn out that any combination of diagrams satisfying these  properties does arise from some $0$-cocycle---see Sections~\ref{sec:4Tin}~and~\ref{sec:KI}.
First, the weight of a diagram is zero whenever it has an \textit{isolated chord} (a chord with an empty arc of the base circle between its endpoints). This is because the two singular knots obtained by resolving an isolated double point are isotopic. This property is called the \textit{one-term relation} ($1$T).  Consider now four singular knots which differ only as displayed in Figure~\ref{fig:4Tout}. Applying the derivative formula once for each shows that there is a relation between the weights of the four corresponding strata. When it comes to a weight system, this can be written in terms of chord diagrams as in Figure~\ref{fig:4Tdiag}. This is the \textit{four-term relation} ($4$T).

\begin{figure}[h!]
\centering
\includegraphics[scale=0.3]{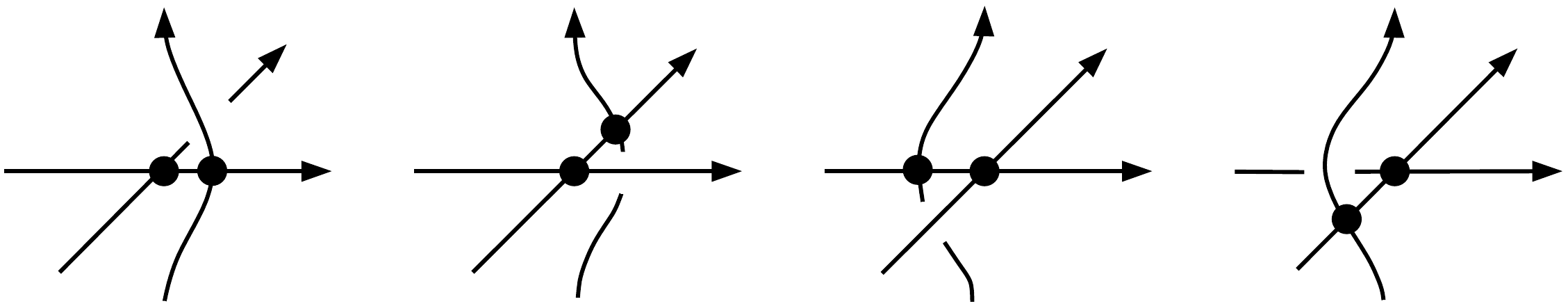}
\caption{Four singular knots that are identical outside the visible part. Any invariant's derivative, at any order, assumes values that are related on four such knots.}\label{fig:4Tout}
\end{figure}

\begin{figure}[h!]
\centering
\labellist
\small\hair 2pt
\pinlabel $-$ at 167 64
\pinlabel $+$ at 366 64
\pinlabel $-$ at 565 64
\pinlabel $=0$ at 764 64
\endlabellist\includegraphics[scale=0.4]{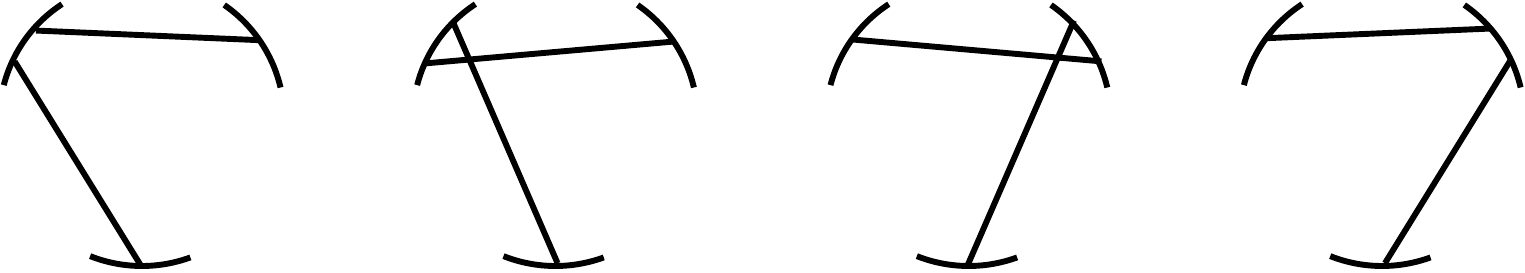}
\caption{These four chord diagrams are identical outside the visible parts. Here the point at infinity can be in any of the three unseen arcs of the circle. For now, this relation is a shorthand for saying that the weight system of a $0$-cocycle is orthogonal to all such linear combinations for the Kronecker pairing. When it comes to the settings of the Kontsevich integral, which are dual to Vassiliev's, this relation will actually mean that the linear combination on the left is set to zero in the algebra of chord diagrams.}\label{fig:4Tdiag}
\end{figure}

\subsubsection{The four-term relation seen from the inside}\label{sec:4Tin}

Here we present the $4$T relation from a different perspective, as it was originally discovered by Birman--Lin in \cite{BirmanLin}, where they show that this relation entirely yields the first sheet of Vassiliev's spectral sequence in cohomological degree $0$.

One very useful feature of the resolved discriminant $\hs$ is that it is built in such a way that the artificial cells themselves can be naturally represented by some kind of chord diagrams, where one can read directly

\begin{itemize}
\item the \textit{complexity} of the cells, which measures how \enquote{deep} they are in $\hs$;
\item the topological orientation of the cells, as well as their boundary expressed in terms of diagrams as well;
\item their \textit{effective codimension}, by which we mean the degree of the cohomology classes where they could potentially participate in a principal part.
\end{itemize} 

At this point it will help to know more about the artificial simplices that define the cells of $\hs$. The artificial simplex that is used to expand a stratum defined by some chord diagram has one vertex for each chord of the diagram. Thus, an ordinary double point contributes $+1$ to the codimension of a stratum and $+1$ to the dimension of the associated simplex, which overall doesn't affect the dimension of the product. However, 
a stratum of knots with triple points or higher has a much larger codimension in $\K$ than the number of chords needed to define it. To compensate for this and ensure that each stratum is expanded to a dimension no less than that of $\Sigma$, Vassiliev adds chords that bring redundant information (don't change the associated stratum of $\K$) and yet raise the dimension of the simplex. Diagrams with less redundant information are still used, to denote faces at the boundary of the bigger cell.

For instance, the diagram on Figure~\ref{fig:exTri} corresponds to a stratum of codimension $3$ in $\K$, and the extra chord gives the corresponding cell a dimension equal to that of $\Sigma$, which allows it to enter the principal part of a $0$-cocycle. An ordinary chord diagram enhanced with such a triangle, whose vertices are away from the other chords, is called a \textit{$\nabla\!$-diagram}.


\begin{figure}[h!]
\centering
\includegraphics[scale=0.4]{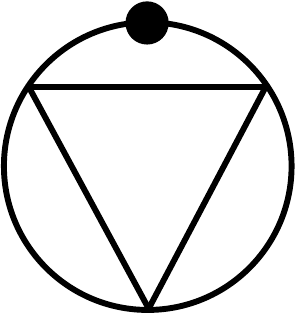}
\caption{This $\nabla\!$-diagram represents the cartesian product of  a $2$-simplex and the stratum of all maps whose only singularity is a triple point. One of its chords contains redundant information, since $K(a)=K(b)$ and $K(b)=K(c)$ already imply $K(a)=K(c)$. This extra chord is there only to raise the dimension of the simplex by $1$. Such a diagram may enter the principal part of a $0$-cocycle, while its faces (obtained by deleting one chord from the triangle) may enter the principal part of a $1$-cocycle.}\label{fig:exTri}
\end{figure}

The depth function on $\hs$ defines a filtration $\Sigma\equiv \hs_1\subset\hs_2\subset\ldots \subset\hs$. The associated spectral sequence computes all cycles of finite complexity, and the relative cycles in the first sheet of this spectral sequence are their potential principal parts. Let us compute this first sheet in cohomological degree $0$. Two kinds of diagrams can enter the principal part of a $0$-cocycle: ordinary chord diagrams (which make up the weight system of the cocycle), and $\nabla\!$-diagrams. 

Vassiliev shows that the boundary of a $\nabla\!$-diagram reduces to the part explained above:
$$\partial_\nabla\left({\raisebox{-0.4em}{\includegraphics[scale=0.12]{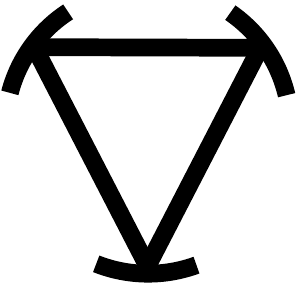}}}\right)={\raisebox{-0.4em}{\includegraphics[scale=0.12]{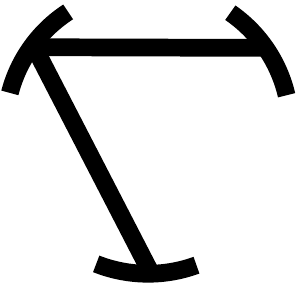}}}-{\raisebox{-0.4em}{\includegraphics[scale=0.12]{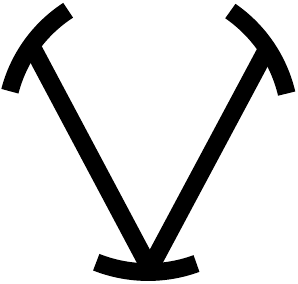}}}+{\raisebox{-0.4em}{\includegraphics[scale=0.12]{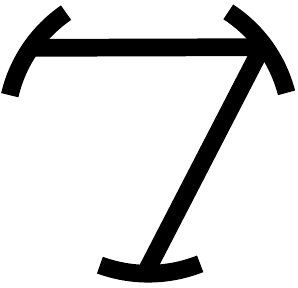}}}$$
The point at infinity is here assumed to be in the top (unseen) arc of the circle for the signs to be correct.

An ordinary chord diagram $D$ has two kinds of boundary pieces. First, 
$$\partial_\bullet D = \sum_{\stackrel{\raisebox{-0.4em}{\includegraphics[scale=0.12]{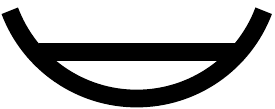}}}{\text{ isolated in $D$}}}{\raisebox{-0.4em}{\includegraphics[scale=0.12]{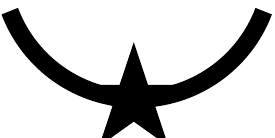}}}$$
corresponds to boundary pieces where an isolated double point tends to become a $1-1$ point where the derivative vanishes. Strata of maps whose derivative vanishes at some point along a given arc are represented by diagrams with a star on that arc. Such pieces of boundary are dead ends---each occurs in the boundary of only one diagram. Therefore, $\Ker \partial_\bullet$ encodes exactly the $1$T relation. 

Second, 
$$\partial_V D = \sum_{\raisebox{-0.4em}{\includegraphics [scale=0.09]{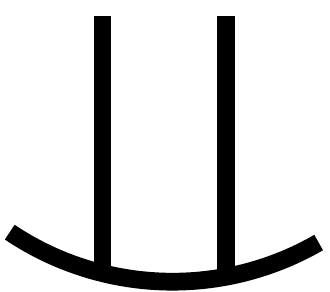}}}{\pm\,  \raisebox{-0.4em}{\includegraphics [scale=0.12]{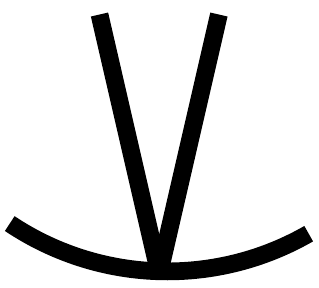}}}$$
where the sum is over pairs of neighboring endpoints of distinct chords in $D$, which may or may not graphically intersect. Such boundary pieces correspond to two double points merging into a triple point.


A linear combination $w$ of ordinary chord diagrams satisfying the $1$T relation can be completed (by $\nabla\!$-diagrams) into a cycle in the first sheet if and only if $\partial_V w$ lies in $\Im \partial_\nabla$. From the definition of $\partial_\nabla$, one sees that $$\Im \partial_\nabla = \Span^\perp\left\lbrace\qtcat , \qtcatt \right\rbrace$$
therefore
$$\partial_V w\in \Im \partial_\nabla \Longleftrightarrow w\in \left[\partial_V^\star \Span\left\lbrace\qtcat , \qtcatt \right\rbrace\right]^\perp$$

Now the direct image of $\Span\left\lbrace\qtcat , \qtcatt \right\rbrace$ by the dual map $\partial_V^\star$ happens to be exactly the subspace spanned by $4$T relators. This shows that the  $1$T and $4$T relations are not only necessary as seen in Subsection~\ref{sec:exuki}, but also sufficient for a linear combination of ordinary chord diagrams to survive in the first sheet of Vassiliev's spectral sequence. This result is due to Birman--Lin~\cite[Lemmas $3.4$ and $3.5$]{BirmanLin}. The main difficulty, actually completely concealed here, was to realize that appropriate orientations of the cells of $\hs$ could be defined so as to make the combinatorics of the boundary maps much simpler---the $4$T relation is not only short, but it is also free from any external data. This result was later improved by Kontsevich, who showed that the spectral sequence degenerates at the first sheet, meaning in particular that $1$T and $4$T relations are necessary and sufficient for a combination of diagrams to actually be the weight system of some $0$-cocycle. We investigate this result in Section~\ref{sec:KI}.

For some purposes, this approach to the $4$T relations is much more powerful than the popular version from Section~\ref{sec:exuki}. In particular, it generalizes to the study of higher cocycles, which is the point of this paper.
%
%
%
%

\subsection{The Kontsevich integral}\label{sec:KI}

To the naive topologist, the Kontsevich integral is a powerful device to capture self-linking information about a knot, and give yet another striking perspective on the $4$T relation and Vassiliev $0$-cocycles, as it encompasses them all. We will stick to this point of view here. A deeper exposition is possible via Drinfeld's associator and quasi-Hopf algebras theory, see~\cite{DrinfeldQHKZ, drinfeld1990quasitri}. Yet another perspective, which is only conjecturally equivalent, is via perturbative Chern--Simons gauge theory---see for instance \cite{labastida1998kontsevich, lescop2002configuration}.

Unless otherwise stated, results from this section are due to Kontsevich~\cite{Kontsevich}. For an excellent and more thorough introduction, we refer to \cite{BarNatanVKI, CDM, LescopK}.

\subsubsection{The case of braids}\label{subsec:KIbraids}
A \textit{braid} is a finite collection of smooth maps $z_k\colon [0,1]\to \C$ for $k=1\ldots p$, called \textit{strands}, such that $z_k(t)\neq z_l (t)$ for any $k\neq l$ and $t\in [0,1]$, and such that $\{z_k(0)\}_{k=1}^ p=\{z_k(1)\}_{k=1}^ p=\{1,\ldots p\}$. A braid with $p$ strands can be regarded as a path in $\C^p \setminus \Delta$, where $\Delta$ is the \textit{fat diagonal}, made of all points in $\C^p$ where at least two coordinates match. Such a path is a loop only if $z_k(0)=z_k(1)$ for all $k$, when the braid is said to be \textit{pure}.

Vassiliev's theory for long knots (Section~\ref{sec:VC}) can be adapted to braids: chord diagrams are here based on $p$ vertical segments labelled from $1$ to $p$, with a horizontal chord for each double point. Weight systems are defined similarly and play the same role, ruled by a natural adaptation of the $4$T relation, see Figure~\ref{pic:4Tbraid}.

\begin{figure}[!ht]
\labellist
\small\hair 2pt
\pinlabel $-$ at 177 87
\pinlabel $=$ at 418 87
\pinlabel $-$ at 658 87
\pinlabel $=$ at 899 87
\pinlabel $-$ at 1140 87
\endlabellist
\centering 
\hspace{5pt}
\includegraphics[scale =0.282]{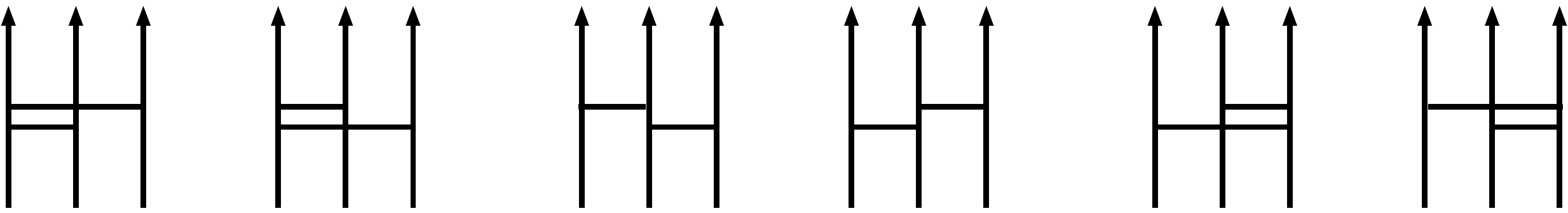}
\caption{The $4$T relations for braids. There can be unseen chords  or vertical segments, but these must be identical in all $6$ diagrams and there can be no chords at an intermediate altitude between the visible ones. These relations are naturally compatible with the $4$T relations for ordinary chord diagrams, under the operation of gluing the top of each vertical segment to the bottom of the next one. In an algebraic framework, one can write these relations as $[H_{ij}+H_{jk},H_{ik}]=0$ for all $i,j,k$, where $H_{ij}$ stands for the diagram with one chord between segments $i$ and $j$, $[,]$ is a commutator, and the multiplication is given by stacking the right diagram on top of the left.}
\label{pic:4Tbraid}
\end{figure}

Given a braid with two strands $z_1$ and $z_2$, the \textit{linking number} $\frac{1}{2 i \pi}\int_0^1 \frac{dz_1-dz_2}{z_1-z_2}$ is an isotopy invariant (because the $1$-form inside is closed). It is an integer when the braid is pure (the algebraic number of times one strand winds around the other) and a half-integer otherwise. It is a complete invariant of $2$-strand braids, as the isotopy class can be recovered from the linking number. 

An isotopy between braids, from the point of view of paths in $\C^p \setminus \Delta$, is simply a path homotopy. A powerful idea to build such homotopy-invariant integrals on paths, extensively described by Chen in  \cite{chen1977}, is to consider \textit{iterated integrals}---that is, where the integration domain is a simplex of the form $0<t_1<\ldots < t_n<1$. 
Thus, to every choice of $n$ pairs of strands given in a particular order, one associates the complex number
\begin{equation} \frac 1 {\left(2i\pi\right)^n}\int_{0<t_1<\ldots < t_n<1}\omega_{i_1, j_1}\wedge\ldots \wedge\omega_{i_n, j_n} \label{eq:int}\end{equation}
where $\omega_{i,j}$ denotes the $1$-form $\frac{dz_i-dz_j}{z_i-z_j}=d\log(z_i-z_j)$ on $\C^p \setminus \Delta$, and where it is understood that each $\omega_{i_k, j_k}$ is pulled back to the parameter simplex $\left\{0<t_1<\ldots < t_n<1\right\rbrace$ via the parametrization of the braid and the projection onto the $k^\text{th}$ coordinate. Each $1$-form measures a linking number, but overall in a nonindependent manner due to the integral bounds. Each choice of $n$ pairs of strands is remembered in the form of a chord diagram based not on a circle but on $p$ vertical segments, with one horizontal chord for each pair; see Figure~\ref{pic:braidex}. 

\begin{figure}[!ht]
\labellist
\small\hair 2pt
\pinlabel $1$ at 285 10
\pinlabel $2$ at 397 10
\pinlabel $3$ at 510 10
\pinlabel $1$ at 967 5
\pinlabel $2$ at 1060 5
\pinlabel $3$ at 1154 5

\pinlabel $1$ at -25 567
\pinlabel $t_2$ at -25 414
\pinlabel $t_1$ at -25 270
\pinlabel $0$ at -25 46

\endlabellist
\centering 
\hspace{5pt}
\includegraphics[width =9.5cm]{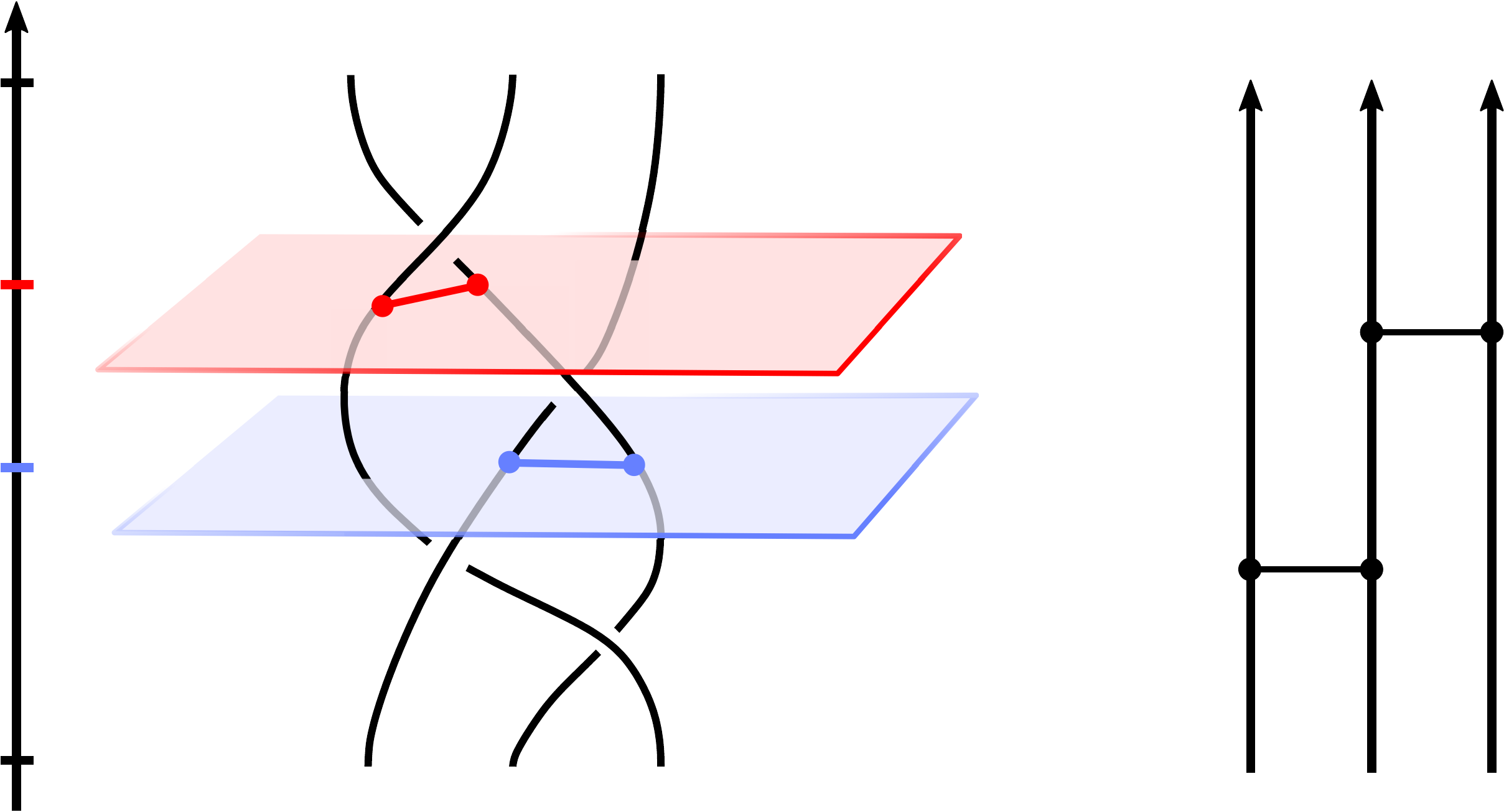}
\caption{On the left, a braid with an ordered choice of two pairs of strands. The variables are bound to the condition $t_1<t_2$.
This choice yields the complex number $\frac{1}{(2i\pi)^2}\int_{0<t_1<t_2<1}d\log(z_1-z_2)\wedge d\log(z_2-z_3)$, which is then stuck to the chord diagram displayed on the right to remember what choice led to it.}
\label{pic:braidex}
\end{figure}

The Kontsevich integral (for braids) is essentially the collection of all these iterated integrals, together with an explicit description of what linear combinations of them yield isotopy invariants. Namely, a linear combination of integrals yields an invariant if and only if the corresponding combination of diagrams satisfies the $4$T relation. Invariants obtained this way turn out to be exactly the collection of all $\Q$-valued Vassiliev invariants of braids---which together form a complete invariant (a result of Kohno  \cite{kohno1987monodromy} proved independently by Bar-Natan \cite{bar1995vassiliev} and later by Papadima over $\Z$ \cite{papadima}).

\subsubsection{Technical details and the case of long knots}\label{sec:techKI}

Let us call a \textit{Morse knot} any long knot $K:\R\rightarrow \R^3\simeq \C_z\times\R_t$ for which the projection $\C_z\times\R_t \rightarrow \R_t$ induces a Morse function. For a given point in $\R^3$, the value of the projection is called its \textit{altitude}. We call a \textit{strand} of the knot the restriction of $K$ to a maximal interval where it has no critical Morse points. By compactness, Morse knots have finitely many strands.

For Morse knots, integrals of the form~\eqref{eq:int} still make sense with a few adjustments. The integration domain has now additional constraints as each variable $t_k$ is limited by the intersection of the lifespans of the two chosen strands. Furthermore, pulling back the $1$-forms $\omega_{i,j}$ yields a $-1$ factor for each $z_i$ or $z_j$ lying on a strand for which the parametrization does not respect the orientation of the $\R_t$ axis. Finally, the chord diagram associated to a choice of $n$ pairs of strands is now an ordinary chord diagram as introduced in Subsection~\ref{sec:exuki}---see Figure~\ref{pic:knotex}.

\begin{figure}[!ht]
\labellist
\small\hair 2pt

\pinlabel $m$ at  -30 140
\pinlabel $t_1$ at -30 250
\pinlabel $t_2$ at -30 368
\pinlabel $t_3$ at -30 457
\pinlabel $M_1$ at -35 529
\pinlabel $M_2$ at -35 575

\endlabellist
\centering 
\hspace{5pt}
\includegraphics[width=10.7cm]{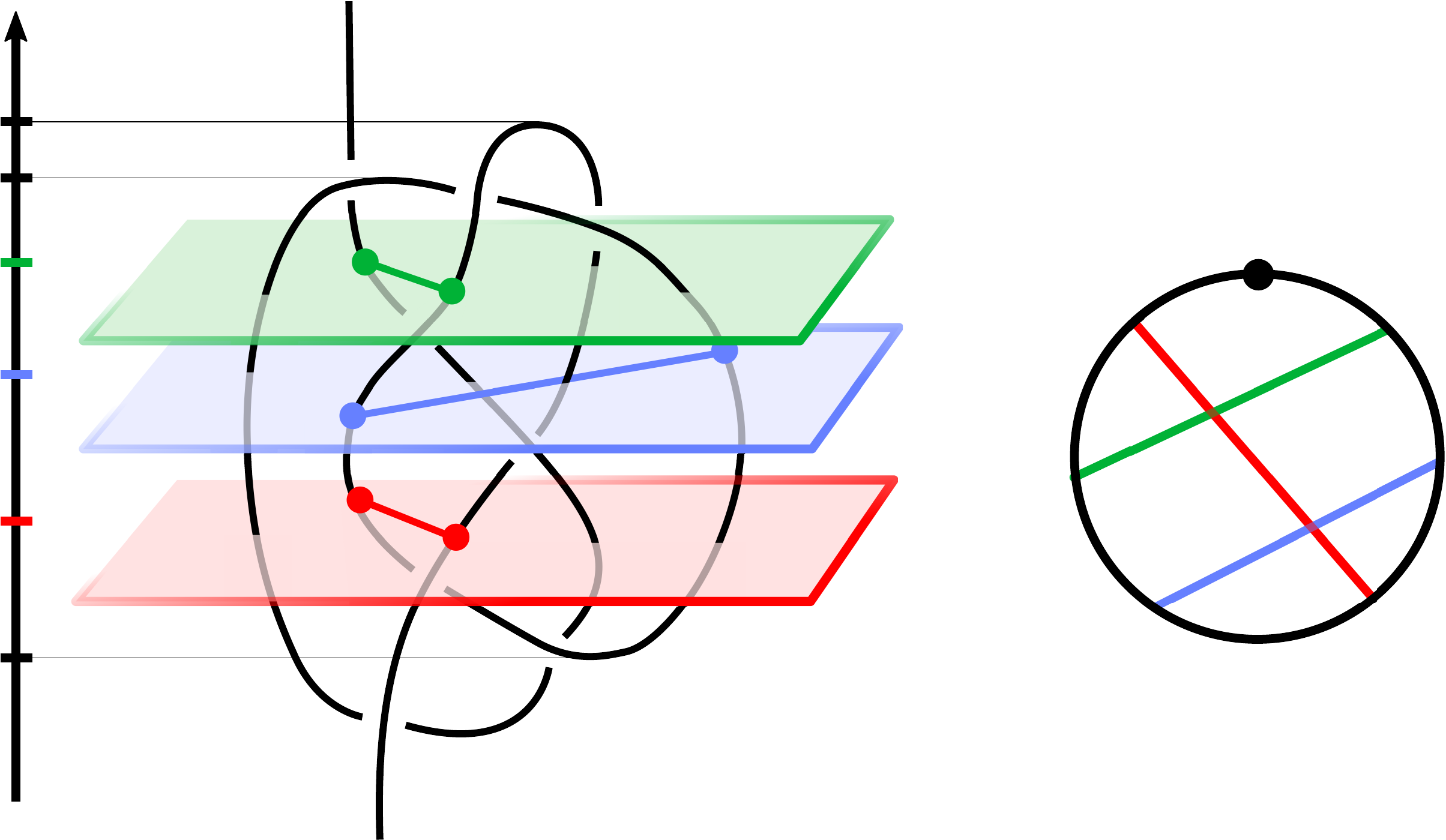}\vspace{5pt}\caption{On the left, a knot with an ordered choice of three pairs of strands. Altitudes are not only bound to the condition $t_1<t_2<t_3$, but also $t_i>m$ for all $i$, as well as $t_2<M_1$ and $t_3<M_2$ in order for the chords to exist.
On the right, the corresponding chord diagram has no isolated chords, which ensures that the integral converges. One can see indeed that when $t_2$ goes near $m$, the second chord brings a singularity $\int \frac{dz}{z}$ with $z$ getting close to $0$, but simultaneously $t_1$ is bound to the interval $(m,t_2)$ whose size tends to $0$.
The colors in the chord diagram are there for the sake of clarity only, they are not part of the structure.}
\label{pic:knotex}
\end{figure}

\subsubsection*{Convergence}

The integral \ref{eq:int} converges in the case of Morse knots as soon as the associated chord diagram has no isolated chords. This is because an isolated chord will in some cases bring a $1$-form $\frac{dz_i-dz_j}{z_i-z_j}$ where $z_i-z_j$ tends to zero with nothing to compensate for it.
When no chords are isolated, such small denominators are always compensated by the short lifespan of some other chords---the integrand may tend to infinity but the volume of the integration domain tends to $0$ faster---see Figure~\ref{pic:knotex}.

The collection of these integrals is gathered as 
\begin{equation}\sum_{n=0}^{\infty} \sum_{\substack{\text{$M$ = admissible}\\ \text{ordered choice of $n$}\\ \text{pairs of strands}}}\left[ \frac {(-1)^{\downarrow(M)}} {\left(2i\pi\right)^n}\int_{\raisebox{-10pt}{$\substack{t_1<\ldots < t_n\\ \text{$t_k\in $ common lifespan}\\ \text{of the two strands}\\ \text{at level $k$}}$}}\omega_{i_1, j_1}\wedge\ldots \wedge\omega_{i_n, j_n}\right]\cdot D_M \label{eq:int2}\end{equation}

where $D_M$ is the ordinary chord diagram yielded by the choice $M$ and \mbox{$\downarrow\!(M)$} stands for the number of times a decreasing strand was chosen. A choice of $n$ pairs is admissible if the lifespans of the two strands in each pair overlap, in a way that is at least partially compatible with the condition $t_1<\ldots < t_n$ (so that the integration domain is non empty when $n>0$), and if the resulting diagram $D_M$ has no isolated chords. By convention, $n=0$ contributes a chord diagram with no chords (with coefficient $1$).

The series \eqref{eq:int2}, which we denote by $\uZ(K)$, lives in $\D^0$, the algebraic completion of $\Span_\C \left\{\text{ordinary chord diagrams}\right\}$. Although it cannot involve diagrams with isolated chords by construction, it is still common to regard $\uZ(K)$ as an element of $\D^0/\left\{1\text{T}\right\}$ since it is ultimately meant to be dual to weight systems.

\subsubsection*{Invariance}

The first outstanding result by Kontsevich is that for any two knots $K$ and $\tilde{K}$ that are isotopic as Morse knots, the difference $\uZ(K)-\uZ(\tilde{K})$ lies in the subspace spanned by the $4$T relators (Figure~\ref{fig:4Tdiag}). In other words, we have a Morse knot invariant $$Z(K)=\left[\uZ(K)\right]\in \A^0=\D^0/\left\{1\text{T}, 4\text{T}\right\}.$$

The proof relies on Stokes' theorem together with a key property of the $1$-forms $\omega_{i,j}=\frac{dz_i-dz_j}{z_i-z_j}$ called \textit{Arnold's identity}, which can be written as:

\begin{equation}\label{Arnold}
\omega_{i,j}\wedge\omega_{j,k} + \omega_{j,k}\wedge\omega_{k,i} + \omega_{k,i}\wedge\omega_{i,j}=0
\end{equation}


Remarkably, $\A^0$ endowed with the natural concatenation of diagrams is an associative and commutative (!) algebra (\cite[Theorem 7]{BarNatanVKI}), and $Z$ is multiplicative with respect to this operation and the connected sum of knots. As a result, one can understand exactly how $Z$ fails to be a classical knot invariant: the addition of two critical points amounts to a connected sum with a \textit{hump}---a trivial knot which we will denote by the symbol $\infty$, see Figure~\ref{pic:hump})---and therefore if we let $c(K)$ denote the number of critical points of $K$, then 
$$\hat{Z}(K)=\frac{Z(K)}{Z(\infty)^{c(K)/2}}$$
defines a classical knot invariant. Note that $Z(\infty)$ is invertible in $\A^0$ because the series \eqref{eq:int2} always starts by the unit (the empty chord diagram) by convention.
\begin{figure}[!ht]
\hspace*{-1cm}\labellist
\small\hair 2pt

\pinlabel $\C_z$ at 548 55
\pinlabel $0$ at 270 110
\pinlabel $\R_t$ at 340 495
\endlabellist
\centering 
\hspace{5pt}
\includegraphics[scale =0.24]{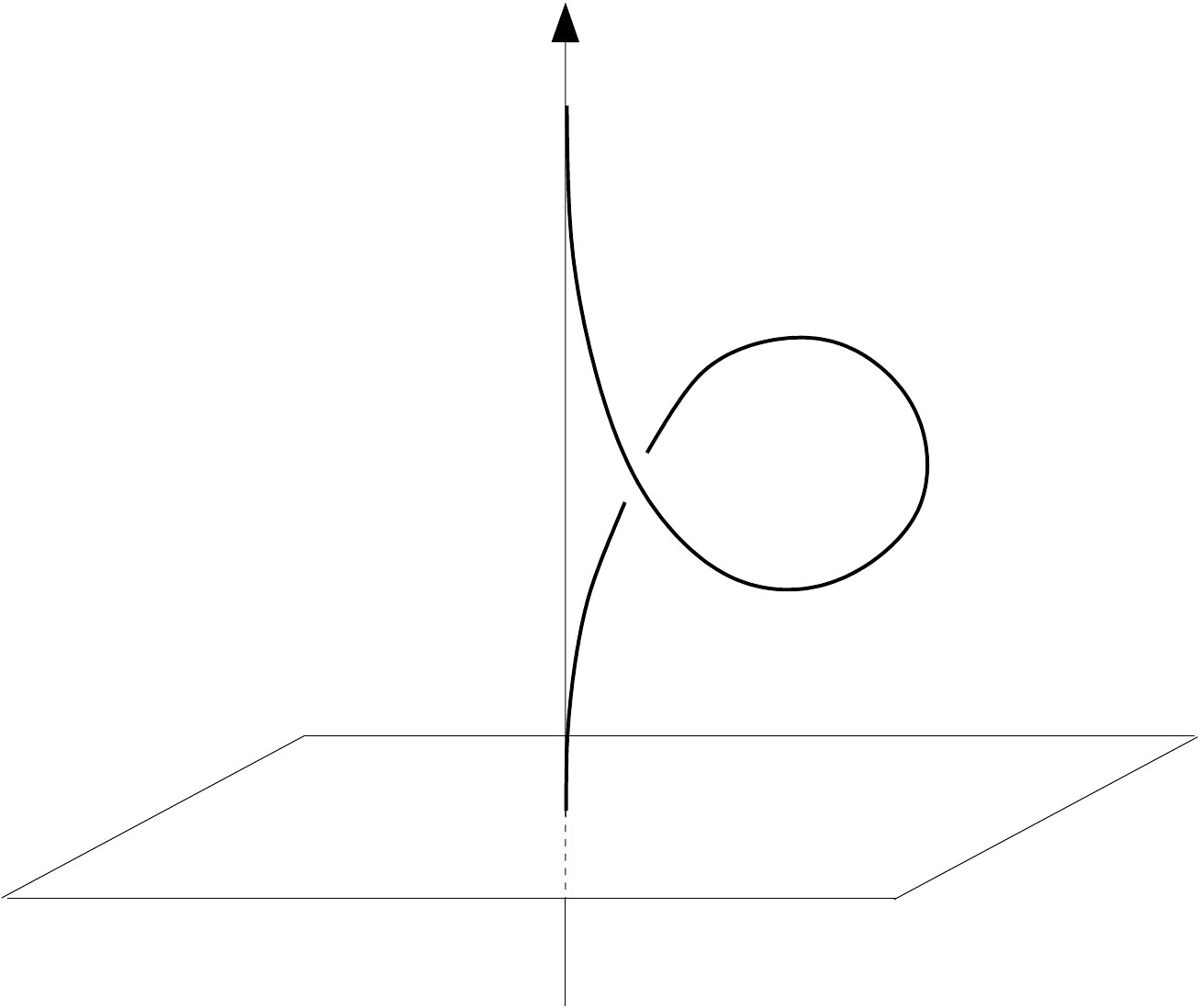}
\caption{The hump, a Morse (un)knot with two critical points}
\label{pic:hump}
\end{figure}

\subsubsection*{Duality with weight systems}

Since Vassiliev--Birman--Lin's weight systems span the orthogonal subspace to $\left\{1\text{T}, 4\text{T}\right\}$ in $\D^0$, they form the dual space to $\A^0$ and may be evaluated as functionals on $\hz(K)$. Kontsevich's major result is then that for any weight system $w$, the $\C$-valued knot invariant $w\left(\hz(\cdot)\right)$ is of finite complexity and has weight system $w$. This implies that in cohomological degree $0$, Vassiliev's spectral sequence degenerates at the first sheet (\enquote{every weight system integrates}). An unpublished result by Kontsevich states that over $\C$ this actually holds in every cohomological degree.

\section{The space of $V$-diagrams $\A^1$}\label{sec:A}

In this section we introduce the target space of our $1$-cocycle. Similarly to $\A^0$ (see Section~\ref{sec:techKI}), it is defined as a quotient of a graded complex vector space freely generated by diagrams which are introduced thereafter:
$$\A^1=\D^1/\Span\left\{1\text{T}, 2\text{T}, \4\text{T}, 16\text{T}, 28\text{T} \right\}.$$

The relations that yield this quotient come from a rewriting of Vassiliev's spectral sequence as explained in Section~\ref{sec:Vass} and Appendix~\ref{Appen:3}, using similar ideas as Birman--Lin \cite{BirmanLin}. They are also essential in the (well-definedness and) cocyclicity of our invariant, as shown in Section~\ref{subsec:cocy}.

A \textit{chord} is an unordered pair of distinct real numbers.
All kinds of chord diagrams that we are about to define are regarded up to positive homeomorphisms of $\R$, and are graded by the number of chords.

An \textit{(ordinary) chord diagram} is a set of finitely many pairwise disjoint chords.

A \textit{$V$-diagram} is the datum of a chord diagram, together with a \enquote{$V$}: three points in $\R$ distinct from each other and from the endpoints of existing chords, together with a choice of two out of the three possible chords to link them. A chord disjoint from the $V$ is called \textit{ordinary}.

A \textit{$V^2$-diagram} is either of two things:
\begin{itemize}
\item a chord diagram with two additional $V$'s disjoint from each other;
\item a chord diagram with four additional distinct points in $\R$, linked by chords which together form a spanning tree of the complete graph on four vertices.
\end{itemize}

\begin{figure}[!ht]
\centering 
\hspace{5pt}
\includegraphics[scale =0.41]{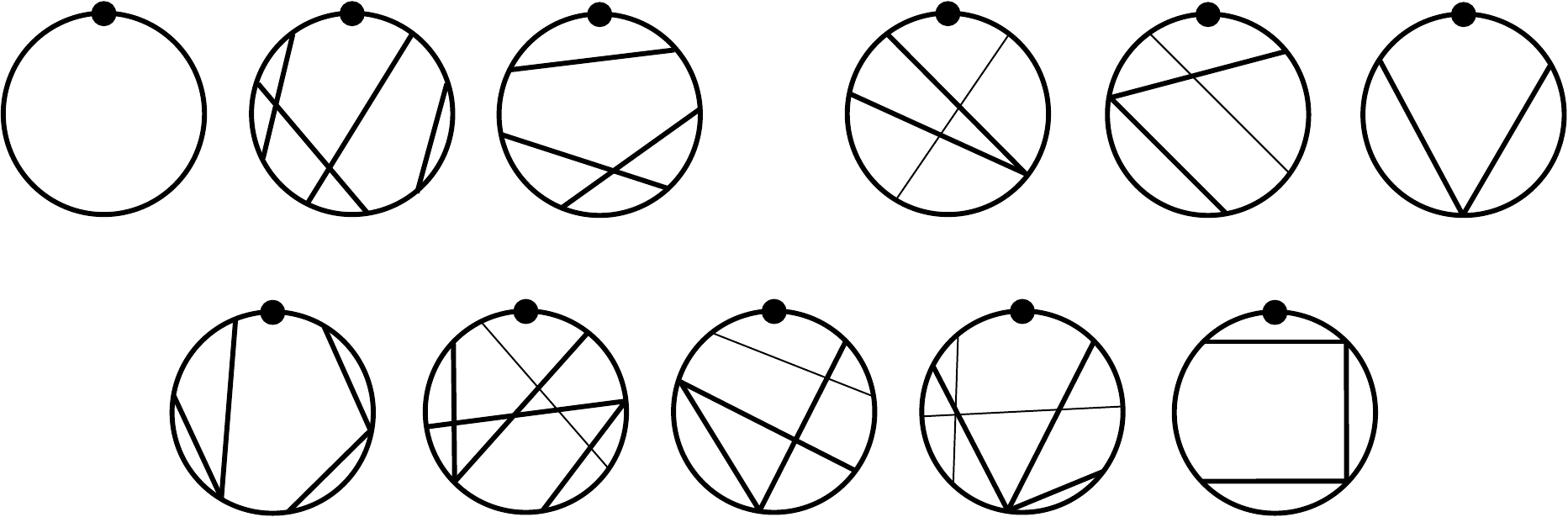}
\caption{Ordinary chord diagrams (top-left), $V$-diagrams (top-right) and $V^2$-diagrams (bottom). Ordinary chords in $V$- or $V^2$-diagrams are drawn thinner for clarity.}
\label{pic:excdvd}
\end{figure}

\begin{remark}
In Vassiliev's seminal paper \cite[Section 2.1]{Vassiliev1990},  $V$-diagrams appear in the boundary of what is defined there as $(3,2,\ldots,2)$-configurations (see the map $\partial_{\nabla}$ recalled here in Section~\ref{sec:4Tin}), while $V^2$-diagrams appear respectively in the faces of $(3,3,2,\ldots,2)$ and $(4,2,\ldots,2)$-configurations.
\end{remark}

The graded completion of the vector space freely generated by $V$-diagrams over $\C$ is denoted by $\D^1$. In the rest of this section, we present the relations that are to be set on $\D^1$, thus defining the quotient $\A^1$.

\subsection{$1$T and $2$T relations}
Recall from Section~\ref{sec:VC} that a chord $\left\lbrace a, b\right\rbrace$ is called isolated when no other chords have endpoints in the interval $(a,b)$. 
A $V$-diagram is set to $0$ when it has an isolated ordinary chord (Figure~\ref{pic:1T2T}a). This is still called the $1$T relation, by analogy with the $1$T relation in $\A^0$.
\begin{remark}
In the settings of the Kontsevich integral, the $1$T relation is sometimes stated in terms of chords that do not intersect other chords, which seems more general than the version with isolated chords. In fact, in $\D^0$ modulo $4$T, the two are equivalent because of \cite[Theorem 7]{BarNatanVKI}\footnote{This theorem implies that up to the $4$T relations, the multiplication of ordinary chord diagrams via concatenation could be just as well defined by cutting one of the two diagrams anywhere along the line and inserting the other diagram there.}. This equivalence is not expected to hold in the present case in $\D^1$. However, be it in $\D^0$ or $\D^1$, neither the good properties of the integral nor Vassiliev's equations actually require more than the version involving isolated chords.
\end{remark}

The $2$T relation affects isolated chords that participate in a $V$. Given a chord diagram and a choice of an endpoint of a chord, there are two ways to attach an isolated chord to this endpoint and form a $V$-diagram. The sum of these two $V$-diagrams is set to $0$ (Figure~\ref{pic:1T2T}b). As usual when several incomplete diagrams are represented within the same equation, it is implied that they are all identical outside the visible area.

\begin{figure}[!ht]
\labellist
\small\hair 2pt

\pinlabel $(\text{a})$ at -10 60
\pinlabel $(\text{b})$ at 360 60
\pinlabel $=0$ at 172 16
\pinlabel $+$ at 591 16
\pinlabel $=0$ at 835 16
\endlabellist
\centering 
\hspace{5pt}
\includegraphics[scale =0.43]{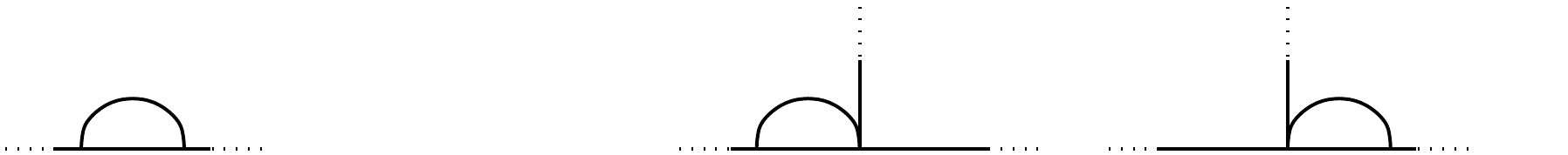}
\caption{$1$T and $2$T relations. The second chord of the $V$ can land anywhere outside the visible area.}
\label{pic:1T2T}
\end{figure}
\subsection{$16$T and $28$T relations}\label{sec:1628}

Let $P$ and $P^\prime$ be two disjoint finite subsets of $\R$, with $P^\prime$ of cardinality $2$, say $P^\prime=\left\lbrace a, b\right\rbrace$. The \textit{linking number} of $P$ and $P^\prime$ is
$$\lk(P,P^\prime)=(-1)^{\hash P\cap [a,b]}.$$

This notion is essential in Birman--Lin's rewriting of Vassiliev's equations in \cite{BirmanLin}. It is also an ingredient of the boundary maps in \cite{MortierCADS} where a combinatorial model is built to calculate part of the cohomology of the space of long knots.

Consider a usual chord diagram and pick four points in $\R\setminus \bigcup \left\lbrace \text{chords}\right\rbrace$, labelled from $1$ to $4$ according to the orientation of $\R$. There are $16$ ways to complete this into a $V^2$-diagram, represented by trees as shown in Figure~\ref{pic:spantrees}.

\begin{figure}[!ht]
\centering 
\hspace{5pt}
\includegraphics[scale =0.4384]{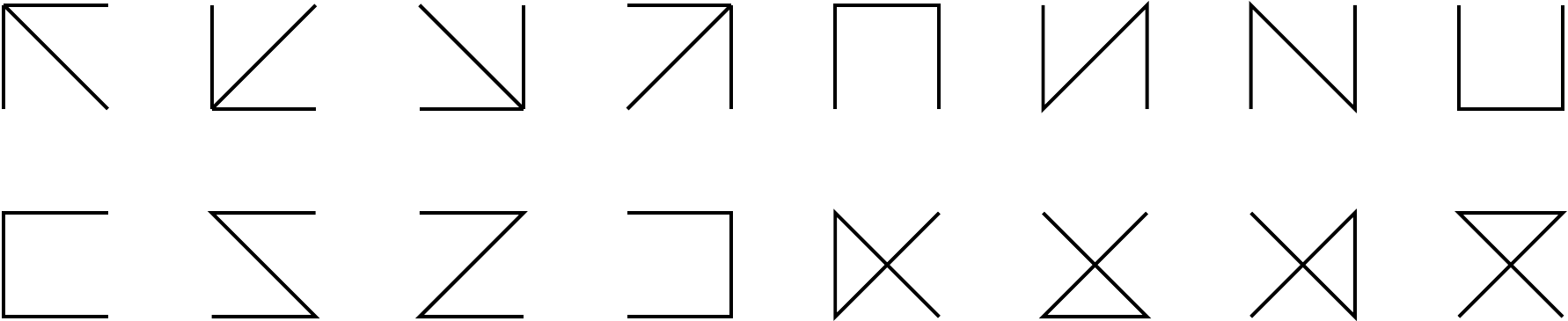}
\caption{The $16$ spanning trees of the complete graph on $4$ vertices. The vertices are labelled as follows: \protect\cgraph}
\label{pic:spantrees}
\end{figure}
By a \textit{desingularization} of a $V^2$-diagram $D$ we mean a $V$-diagram from which one can recover $D$ by shrinking to a point an interval of the base line that contains no endpoints of any chords in its interior.
Each of the $V^2$-diagrams on Figure~\ref{pic:spantrees} can be desingularized in either $6$ ways (for the first four) or $4$ ways (for the remaining twelve), by which the tree is split into a $V$ and an ordinary chord. 

\begin{notation}\label{not:16trees}
By each of the $16$ graphs in Figure~\ref{pic:spantrees} we mean the linear combination of all possible desingularizations of the corresponding $V^2$-diagram, where the coefficient for each summand is the product of two signs:
\begin{itemize}
\item $(-1)^k$ where $k\in \{1,\ldots , 4\}$ is the label of the desingularized vertex;
\item $\lk(P,P^\prime)$ where $P$ and $P^\prime$ are the $V$ and the ordinary chord that result from the desingularization.
\end{itemize}
Two examples are presented in Figure~\ref{pic:expld}.
\end{notation}

\begin{figure}[!ht]
\centering
\labellist
\small\hair 2pt
\pinlabel $\text{Base diagram:}$ at 90 403
\pinlabel $=$ at 175 232
\pinlabel $-$ at 373.5 232
\pinlabel $-$ at 545.5 232
\pinlabel $+$ at 715.5 232

\pinlabel $-$ at 371.5 61
\pinlabel $+$ at 541.5 61
\pinlabel $-$ at 711.5 61
\pinlabel $+$ at 881.5 61
\pinlabel $-$ at 1052.5 61

\pinlabel $=$ at 175 61
\pinlabel $2$ at 223 338
\pinlabel $3$ at 347 338
\pinlabel $1$ at 223 468
\pinlabel $4$ at 347 468

\endlabellist
\hspace{5pt}
\includegraphics[scale=0.318
]{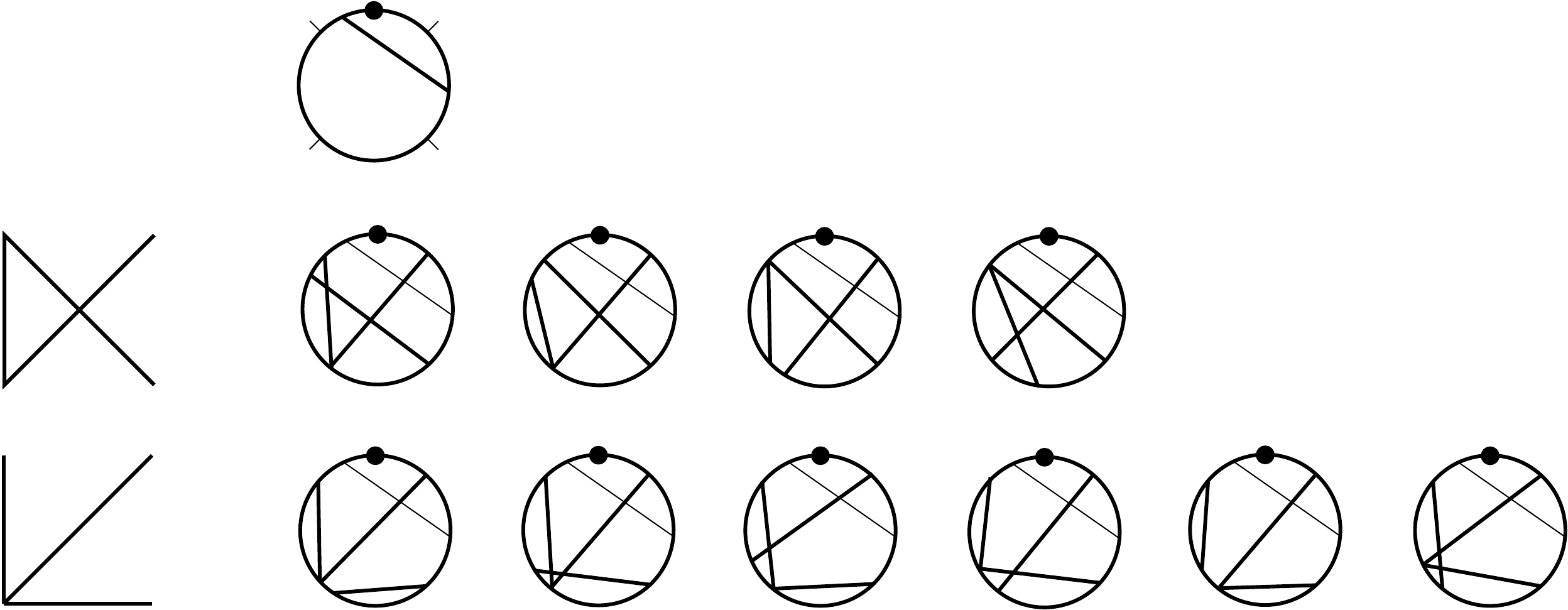}
\caption{Two examples of shortcuts and the linear combinations they stand for, given a base chord diagram with markings, shown at the top. The signs are explained in Notation~\ref{not:16trees}.}
\label{pic:expld}
\end{figure}
%

The $16$T and $28$T relations are as shown in Figures~\ref{pic:16T} and~\ref{pic:28T}. For comparison, the usual $4$T relations with our compact notation are shown in Figure~\ref{pic:4T}. \begin{figure}[!ht]
\labellist
\small\hair 2pt
\pinlabel $+$ at 94 32
\pinlabel $=0$ at 231 32
\pinlabel $+$ at 94 145
\pinlabel $=0$ at 231 145
\endlabellist
\centering 
\hspace{5pt}
\includegraphics[scale =0.44]{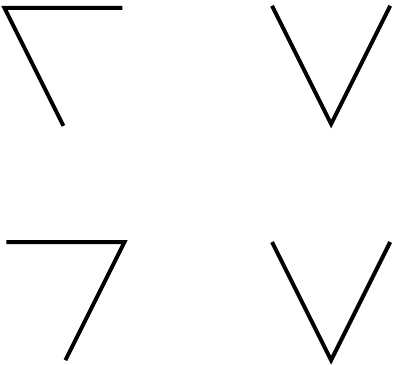}
\caption{The usual $4$T relations for comparison, written in a compact fashion inspired by the approach from Section~\ref{sec:4Tin}. The third one that one would write down is omitted as it is the difference of these two. The vertices are ordered as follows: \protect\cgraphbis}
\label{pic:4T}
\end{figure}
\begin{figure}[!ht]
\labellist
\small\hair 2pt
\pinlabel $+$ at 94 258
\pinlabel $+$ at 223 258
\pinlabel $+$ at 350 258
\pinlabel $=0$ at 487 258
\pinlabel $-$ at 94 145
\pinlabel $+$ at 223 145
\pinlabel $-$ at 350 145
\pinlabel $=0$ at 487 145
\pinlabel $+$ at 94 32
\pinlabel $+$ at 223 32
\pinlabel $-$ at 350 32
\pinlabel $=0$ at 487 32
\endlabellist\centering 
\hspace{5pt}
\includegraphics[scale =0.44]{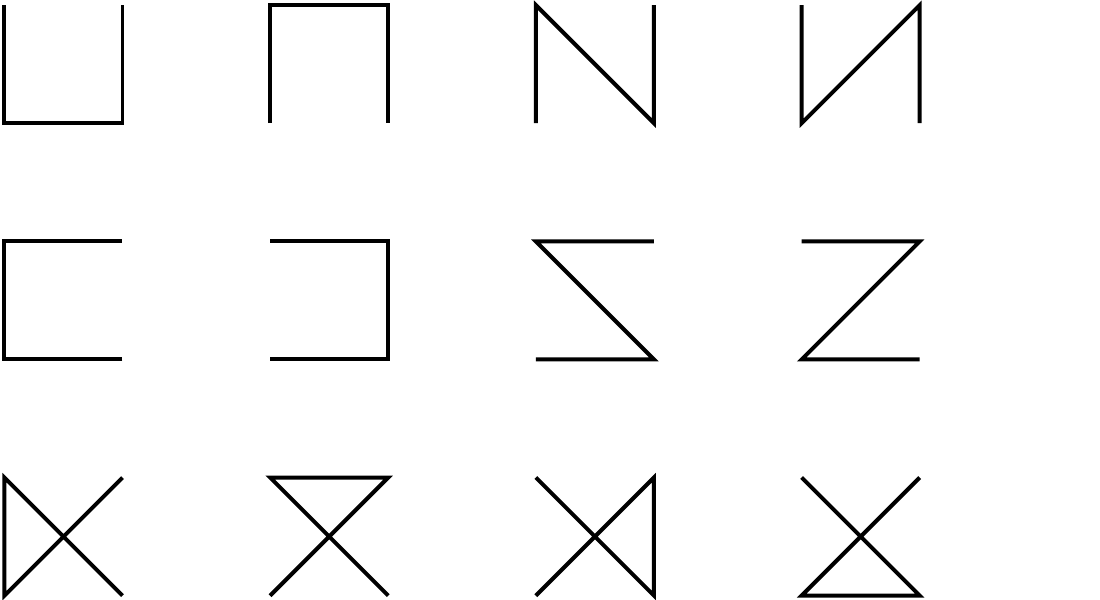}
\caption{The three $16$T relations. Each diagram brings four terms (see Notation~\ref{not:16trees}).}
\label{pic:16T}
\end{figure}

\begin{figure}[!ht]
\labellist
\small\hair 2pt
\pinlabel $+$ at 94 258
\pinlabel $+$ at 223 258
\pinlabel $+$ at 350 258
\pinlabel $+$ at 477 258
\pinlabel $+$ at 605 258
\pinlabel $=0$ at 732 258
\pinlabel $+$ at 94 145
\pinlabel $+$ at 223 145
\pinlabel $+$ at 350 145
\pinlabel $+$ at 477 145
\pinlabel $+$ at 605 145
\pinlabel $=0$ at 732 145
\pinlabel $+$ at 94 32
\pinlabel $+$ at 223 32
\pinlabel $+$ at 350 32
\pinlabel $+$ at 477 32
\pinlabel $+$ at 605 32
\pinlabel $=0$ at 732 32
\endlabellist\centering 
\hspace{5pt}
\includegraphics[scale =0.44]{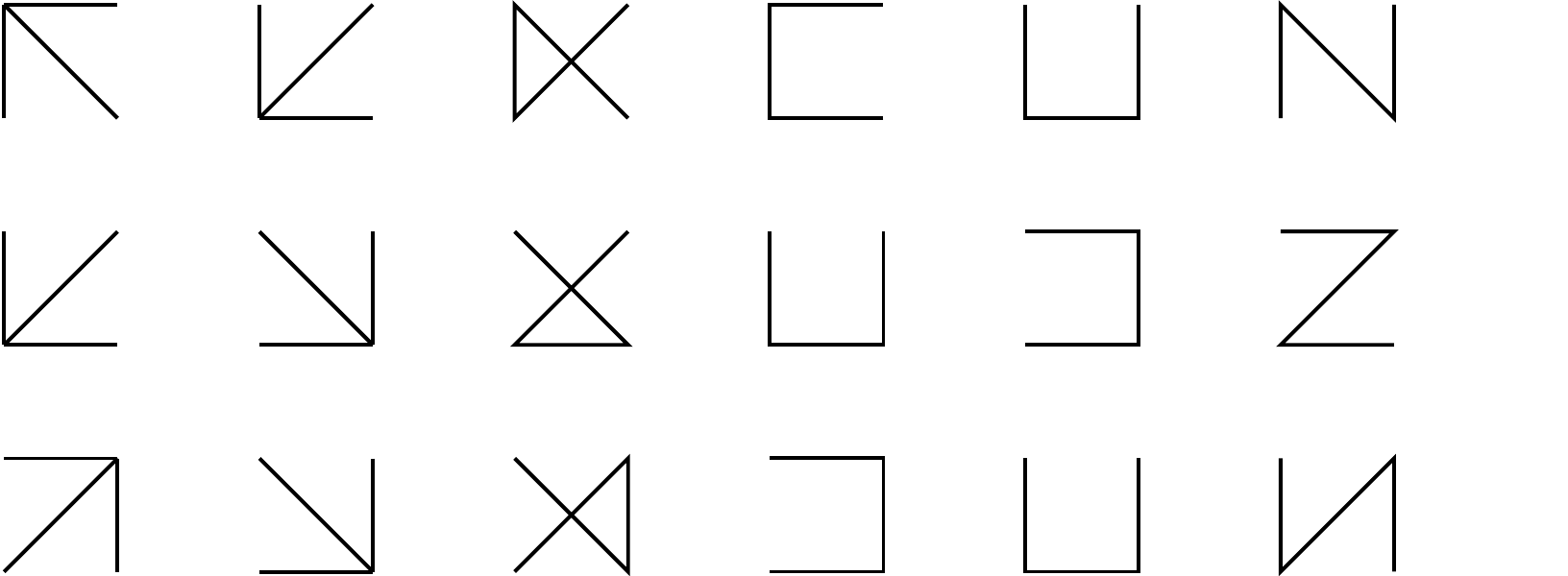}
\caption{The three $28$T relations. In each relation the first two diagrams bring six terms each and the other diagrams four terms each (see Notation~\ref{not:16trees}). Three other similar relations are yielded by these, involving the other three possible pairs of $3$-leaved trees.}
\label{pic:28T}
\end{figure}

\begin{remark}
The three $28$T relations as they are presented here in Figure~\ref{pic:28T} look like they can be obtained from each other by moving the point at infinity (and applying the second $16$T relation where necessary). However, this fails when the relations are presented in extended fashion, with each diagram expanded as in Figure~\ref{pic:expld} and $28$ $V$-diagrams in total, because the linking numbers involved actually depend on the point at infinity.
The fact that the usual $4$T relations can be written without referring to the point at infinity is morally related to the one-to-one correspondence between long knots and compact knots (embeddings of $\S^1$ into $\S^3$) at the level of isotopy classes. There is no such correspondence at the level of higher (co)cycles.
\end{remark}

\subsection{$\4$T relations}\label{sec:4x4}

Consider again a usual chord diagram and now pick two disjoint triples of points in $\R\setminus \bigcup \left\lbrace \text{chords}\right\rbrace$, labelling from $1$ to $3$ the vertices of the triple which owns the least of all six numbers, and from $4$ to $6$ the other three\footnote{This lexicographical order is to be compared with the one that defines the co-orientation of the spaces $\chi(\Gamma^d,J)$, and therefore the ingredient $\zeta$ of the incidence signs in the spectral sequence, in Vassiliev's \cite[Chapter V, Sections $3.3.4$ and $3.3.6$]{VassilievBook}.}, again according to the orientation of $\R$. Each triple can be completed into a $V$ in $3$ different ways, which makes nine different $V^2$-diagrams in total.

The $\4$T relations are then built as follows. Pick any pair of $4$T relators as presented in Figure~\ref{pic:4T}---there are four such pairs, for example \protect\fTone \, first and then \protect\fTtwo . Now expand the formal product of these relators, and set it to $0$: in the example, one obtains \protect\fTexample $=0$. Each of these two-component diagrams stands for the alternating sum of all four ways to desingularize the corresponding $V^2$-diagram into a $V$-diagram, where the coefficient of each summand is the product of\begin{itemize} \item $(-1)^k$, where $k\in \{1,\ldots , 6\}$ is the label of the desingularized vertex according to the scheme \protect\cgraphter  \vspace*{8pt}\item $\lk(P,P^\prime)$ where $P$ and $P^\prime$ are the resulting two ordinary chords.\end{itemize}

%

\section{The integral $Z^1$ for paths of Morse knots}\label{sec:leaf}

Similarly to the Kontsevich integral, our $1$-cocycle is more natural to define first in the space of Morse knots (defined in Section~\ref{sec:techKI}). 

\subsection{Main formula}

For $n\in \N$ with $n\geq 2$, we consider the unbounded closed simplex $$\Delta^n=\left\lbrace (t_1,\ldots,t_n)\in \R^n \,\left|\, t_1\leq\ldots\leq t_n \right\rbrace\right.$$
and its boundary $\partial\Delta^n=\bigcup_{i=1}^{n-1} \Delta^n_i$, where 
$$ \Delta^n_i=\left\lbrace (t_1,\ldots,t_n)\in \Delta^n \,\left|\, t_i= t_{i+1} \right\rbrace\right.$$
The usual orientation of $\R^n$ induces an orientation on $\Delta^n$, which in turn induces an orientation on $\partial\Delta^n$. Each face of $\Delta^n$ can be parametrized by $\Delta^{n-1}$ via duplication of the $i$-th coordinate. When doing so in order to integrate over $\partial\Delta^n$, a sign $(-1)^{i-1}$ appears to account for the orientation---assuming the convention \enquote{outer normal$\,\wedge\, \partial X \equiv X$ } for the orientation of a boundary. 

Let $\mu\colon [a,b]\times\R\rightarrow\R^3\simeq \C\times\R$ be a smooth path in the space of Morse knots. We define $Z^1(\mu)\in\A^1$ by the formula:

\[\sum_{n=1}^{\infty}\dfrac{1}{(2i\pi)^{n+1}}\underset{[a,b]\times \partial\Delta^{n+1}}{\int}\sum_{
\begin{array}{c}
\text{\footnotesize applicable pairings}\\
M=\left\lbrace
\left\lbrace z_j,z_j^\prime \right\rbrace \right\rbrace
\end{array}} (-1)^{\downarrow(M)}D_M\bigwedge_{j=1}^{n+1} \frac{dz_j-dz_j^\prime}{z_j-z_j^\prime}\]

where 
\begin{itemize}
\item for a given $\phi\in \left[a,b\right]$, and a given $t$ on the face $\dni$, an applicable pairing consists of a collection of pairs of complex numbers $z_j\neq z_j^\prime$,  such that \begin{itemize}
\item for every $j$, $(z_j, t_j)$ and $(z_j^{\prime}, t_j)$ lie on the knot $\mu_\phi$,
\item $z_i=z_{i+1}$,
\item $(z_i^\prime,t_i) < (z_{i+1}^\prime,t_{i+1})$ in the order induced by the knot's orientation.\footnote{The last two conditions mean that the two chords at the levels $i$ and $i+1$ should form a $V$, and the corresponding differential forms are ordered lexicographically.}
\end{itemize}

\item $D_M$ is the $V$-diagram naturally corresponding to $M$ and $\mu_\phi$.
\item \mbox{$\downarrow\!(M)$} stands for the number of points $(z_j^{(\prime)}, t_j)$ located on decreasing strands of the knot $\mu_\phi$. The point $(z_i,t_i)=(z_{i+1},t_{i+1})$ contributes only once.

\end{itemize}

\subsubsection*{Some remarks on the definition}
If one rewrites the differential forms in terms of $d\phi$ and the $dt_j$, and then expand the product, it appears that the factor providing the $d\phi$ part has to be either $j=i$ or $j=i+1$ (where $i$ is the index of the face $\dni$), because besides $d\phi$ the two corresponding $1$-forms depend on the same $dt_i$. In other words, only the $1$-forms $dz_j-dz_j^\prime$ corresponding to the $V$ in a $V$-diagram have the ability to measure the linking of two strands both in space and time. 
The integral will vanish on the parts of the integration domain where the $V$ is located on a steady portion of the knot. Ordinary chords will always contribute, the same way they do in the Kontsevich integral, measuring self-linking in space only, as soon as the $V$ is located on a part of the knot that moves with time---see Figure~\ref{pic:pathex}.

The $2$T relation, together with the Arnold identity \eqref{Arnold}, imply that a $V$-diagram with an isolated $V$ (that is, $\{\{a,b\}, \{b,c\}\}$ where $(\min(a,b,c),\max(a,b,c))$ meets no ordinary chords) will never contribute non-trivially to $Z^1$.

\begin{figure}[!ht]
\labellist
\small\hair 2pt
\pinlabel $a$ at 386 980
\pinlabel $b$ at 1073 983
\pinlabel $\R_\phi$ at 1264 983
\pinlabel $m(a)$ at -26 188
\pinlabel $t_1$ at 11 353
\pinlabel $t_2$ at 11 548
\pinlabel $M_1(a)$ at -36 650
\pinlabel $M_2(a)$ at -36 715
\pinlabel $\R_t$ at 6 857
\endlabellist
\hspace{19pt}
\includegraphics[width=12.9cm]{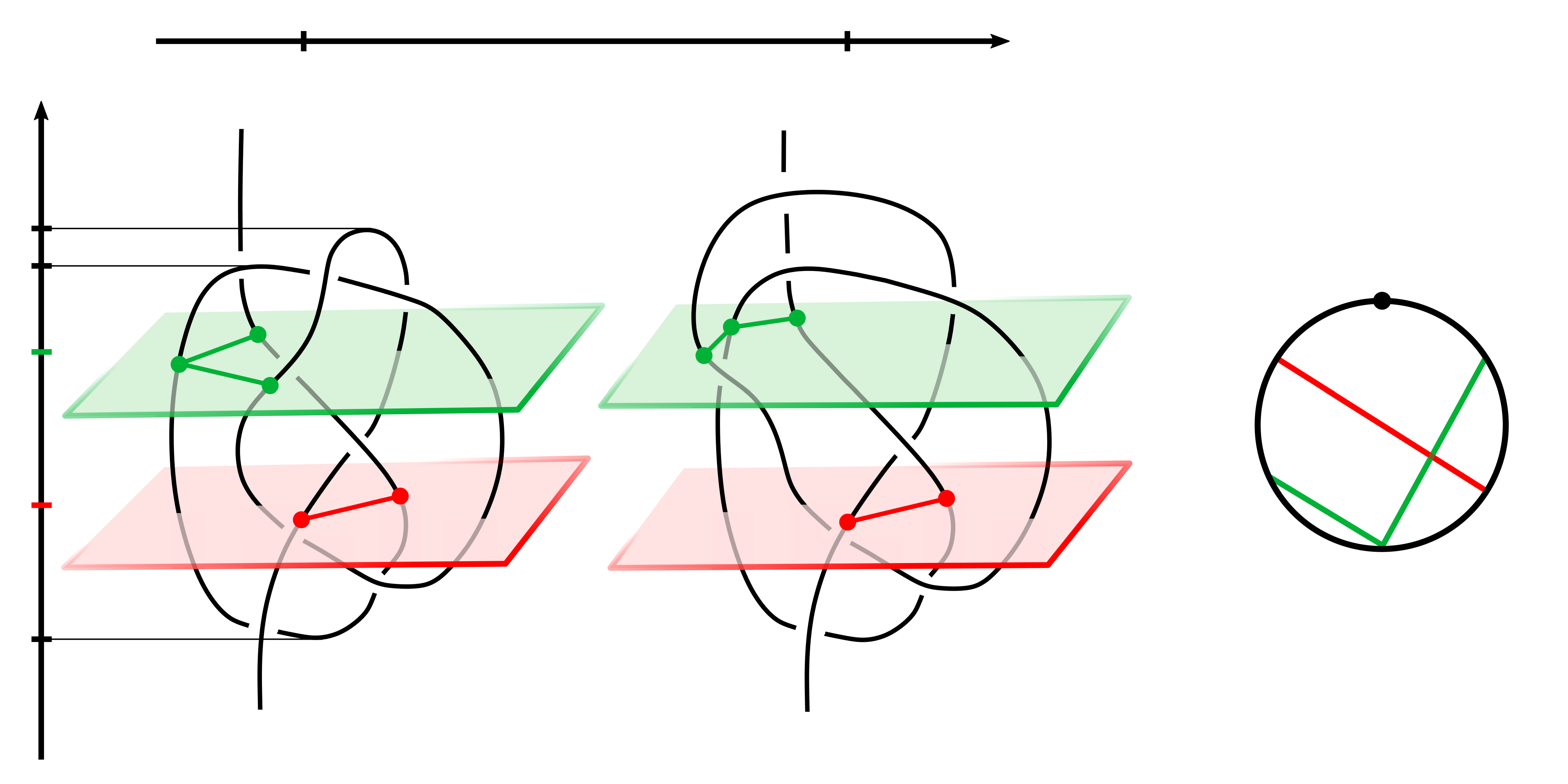}
\caption{On the left, a short path in the space of Morse knots (with only the initial and final knots displayed), and an applicable pairing with $n=2$. The variables are bound to the conditions $a\leq \phi\leq b$ and  $m(\phi)<t_1<t_2<\min(M_1(\phi), M_2(\phi))$. The top two chords continuously measure the linking between the three strands involved, over time ($\phi$) as well as space ($t_2$), while the ordinary chord below is only sensitive to space as in the Kontsevich integral. For lower values of $t_2$, all three strands stand still in time, and the integral formally evaluates to $0$ over this part of the integration domain. On the right, the $V$-diagram corresponding to this particular choice of strands.}
\label{pic:pathex}
\end{figure}

\subsection{A formula for braids}\label{subsec:braids}

If we isolate a chunk of the integration domain, of the form $$\Pi=\left\lbrace (\phi;t_1,\ldots,t_n)\in \left[a,b\right]\times\R^n\left|\begin{array}{c}
\phi_{\min}\leq\phi\leq\phi_{\max}\\t_{\min}\leq t_1<\ldots< t_n\leq t_{\max}

\end{array}\right\rbrace\right.$$ 

so that the path has no critical Morse points between these values, then this part can be seen as a moving braid with loose endpoints:
$$\beta\colon [\phi_{\min},\phi_{\max}]\times[t_{\min},t_{\max}]\rightarrow \C^p\setminus\Delta$$
where $p$ is the number of strands and $\Delta$ the fat diagonal of $\C^p$ (see Section~\ref{subsec:KIbraids}).

Thus we can afford a compact formula \eqref{eq:compact} like the one given by Lescop in her notes on the Kontsevich integral \cite[end of Section~1.4]{LescopK}. Recall the formal Knizhnik--Zamolodchikov connection (see for instance Kohno's book \cite{KohnoCFT} for the context in conformal field theory): $$\Omega_p=\frac{1}{2i\pi}\sum_{1\leq i < j\leq p}\Gamma_{ij}\,\omega_{ij}$$
where the $\Gamma_{ij}$ are the $1$-chord diagrams on $p$ vertical strands and $\omega_{ij}$ is the $1$-form $d\log(z_i-z_j)$ on $\C^p\setminus \Delta$. Similarly, we let $\Gamma_{ijk}$, where $i$, $j$ and $k$ are distinct, stand for the chord diagram on $p$ strands with a chord $\left\lbrace i,j\right\rbrace$ and a chord $\left\lbrace j,k\right\rbrace$ \textit{at the same altitude}. 

The set of pairs $\left\lbrace i,j\right\rbrace$ with $1\leq i\neq j\leq p$ is endowed with the lexicographical order. For $p\geq 3$ we define the $2$-form (see Figure~\ref{pic:Lamb3}) $$\Lambda_p=\frac{1}{(2i\pi)^2}\sum_{\left\lbrace i,j\right\rbrace<\left\lbrace j,k\right\rbrace}\Gamma_{ijk}\,\omega_{ij}\wedge\omega_{jk}$$ 

\begin{figure}[!ht]
\hspace*{-0.85cm}\labellist
\small\hair 2pt
\pinlabel $\omega_{12}\wedge\omega_{13}$ at 140 42
\pinlabel $+$ at 202 42
\pinlabel $\omega_{12}\wedge\omega_{23}$ at 373 42
\pinlabel $+$ at 435 42
\pinlabel $\omega_{13}\wedge\omega_{23}$ at 603 42
\pinlabel $1$ at 13 3
\pinlabel $2$ at 55 3
\pinlabel $3$ at 97 3
\pinlabel $1$ at 246 3
\pinlabel $2$ at 288 3
\pinlabel $3$ at 330 3
\pinlabel $1$ at 477 3
\pinlabel $2$ at 521 3
\pinlabel $3$ at 563 3

\pinlabel $\dfrac{1}{(2i\pi)^2}\left[\vphantom{\frac{\Bigg|}{\Bigg|}}\right.$ at -35 45
\pinlabel $\left.\vphantom{\frac{\Bigg|}{\Bigg|}}\right]$ at 650 45
\endlabellist
\centering 
\hspace{5pt}
\includegraphics[scale=0.53]{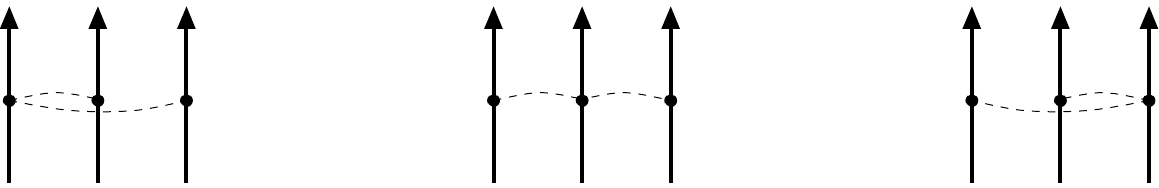}
\caption{The $2$-form $\Lambda_3$}
\label{pic:Lamb3}
\end{figure}

The integral is then defined over $\Pi$ by the formula

\begin{equation}
Z^1(\beta)=\sum_{n=1}^{\infty}\int_\Pi\sum_{i=1}^n (-1)^{i-1}\bigwedge_{j=1}^n\left[(\beta\circ (\Id\times\pr_j))^*\left(\begin{array}{cl}
\Lambda_p&\text{if $j=i$}\\
\Omega_p&\text{otherwise}
\end{array}\right)\right]
\label{eq:compact}
\end{equation}

A version of $16$T, $28$T and $\4$T relations can be naturally defined on $V$-diagrams based on $n$ strands like the $4$T relation from Figure~\ref{pic:4Tbraid}, consistently with the operation of connecting the strands from $1$ to $n$ into a single line.

\begin{lemma}
The integral $Z^1(\mu)$ is absolutely convergent.
\end{lemma}

\begin{proof}
The $\pm\infty$ bounds of the integration domain are not a problem because the range of $t$-values that bring non-trivial contributions is bounded for each knot, and uniformly so since the range of $\phi$ is compact.

The rest of the proof is similar to the case of the Kontsevich integral (see e.g. \cite[Section~4.3.1]{BarNatanVKI}). The only new case consists of a singularity brought by an isolated chord that participates in a $V$. It is solved by the $2$T relation: indeed, when put together, the contributions of the two diagrams become, up to sign:
 $$\int\ldots\wedge \omega_{ij}\wedge (\omega_{jk}-\omega_{ik})\wedge\ldots$$
where $\left\lbrace i,j\right\rbrace$ is the isolated chord. By the Arnold identity \eqref{Arnold}, this amounts to $$\int\ldots\wedge \omega_{ik}\wedge \omega_{jk}\wedge\ldots$$
where both small denominators have disappeared.
\end{proof}

\subsection{Cocyclicity of $Z^1$}\label{subsec:cocy}

This subsection is devoted to the proof of the homotopy invariance of $Z^1$. In terms of connections, while the key point in the invariance of the Kontsevich integral was the flatness of the KZ connection ($d\Omega_p$ and $\Omega_p\wedge\Omega_p$ are both $0$), one of the major ingredients here is the  fact that $\Omega_p\wedge\Lambda_p-\Lambda_p\wedge\Omega_p=0$.

\begin{theorem}
$Z^1$ is a $1$-cocycle in the space of Morse knots.
\end{theorem}
\begin{proof}
Let $H$ be a smooth map $[0,1]\times[a,b]\times\R\rightarrow \C\times\R$ such that for every $\psi\in [0,1]$ and every $\phi\in [a,b]$ the map $H(\psi, \phi, \cdot)$ is a Morse knot, and the knots $H(\psi, a, \cdot)$ and $H(\psi, b, \cdot)$ do not depend on $\psi$. The assertion is that $$Z^1(H(1,\cdot,\cdot))=Z^1(H(0,\cdot,\cdot)).$$

The conditions above imply that all knots $H(\psi, \phi, \cdot)$ have the same number $c$ of Morse critical points and allow us to define $c$ smooth maps of the form $$T\colon [0,1]\times[a,b]\rightarrow \R$$ such that for every $(\psi, \phi)\in [0,1]\times[a,b]$, $T(\psi, \phi)$ is a critical point of $H(\psi, \phi, \cdot)$.

Similarly to the proof of the invariance of the Kontsevich integral, we use Stokes' theorem. On the left-hand side, we have $0$ as we integrate exact forms. On the right-hand side, the boundary of the integration domain is made of the following parts:

\begin{itemize}[leftmargin=*]
\item $\psi=0$ or $\psi=1$: contributes $Z^1(H(1,\cdot,\cdot))-Z^1(H(0,\cdot,\cdot)).$
\item $\phi=a$ or $\phi=b$: no contribution because the two differential forms coming from a $V$ are collinear (both multiples of the same $dt_i$) along these faces since there is no dependence on $\psi$ or $\phi$, so their exterior product vanishes.
\item $t_j=t_{j+1}=t_{j+2}$ for some $j$: no contribution because $\Omega_p\wedge\Lambda_p-\Lambda_p\wedge\Omega_p=0$ (this equality is equivalent to the set of $16$T and $28$T relations, see Appendix~\ref{Appen:1}).
\item $t_j=t_{j+1}$ and $t_k=t_{k+1}$ for some $j$ and some $k$ with $j+1<k$. All the contributions from this stratum cancel out by the $\4$T relations.
\item $t_j=T(\psi, \phi)$, meaning the $j$-th level reaches a critical point:\begin{itemize}

\item if only one branch of the knot near the critical point is involved in the $j$-th graph component, then this piece of boundary cancels off with the similar part of the integral that involves the other branch.
\item if the two branches are involved and the $j$-th component is an ordinary chord, then the contribution vanishes because of the $1$T relation.

\item in case of a $V$ whose tips occupy each one of the branches, there is no contribution since the two chords of the $V$ both bring the same differential form on this piece of boundary.

\item the two remaining cases---of a $V$ one of whose chords links the two branches---can be grouped together thanks to the $2$T relation, and cancel out because of the Arnold identity \eqref{Arnold}.\qedhere
\end{itemize}
\end{itemize}
\end{proof}
Note that all of the defining relations of $\A^1$ are used in this proof.

\subsection{Elementary functoriality}\label{sec:elemf}

The space of knots acts on paths via left and right connected sum: 
if $K_1$ and $K_2$ are (Morse) knots and $\mu$ is a path of (Morse) knots,  then $K_1\hash\mu\hash K_2$ is the path $\mu$ performed with two steady factors $K_1$ and $K_2$ on the left and on the right respectively. Similarly the space $\D^1$ of $V$-diagrams is endowed with the structure of a $\D^0$-bimodule, with operations defined by left and right concatenation. This fails to descend into an $\A^0$-bimodule structure on $\A^1$, because the classical $4$T relations that define $\A^0$ do not hold \textit{a priori} in $\A^1$. However, a weaker version of these relations holds in $\A^1$ as shown by Theorem~\ref{thm:bimod}.

Let $K$ and $\tilde{K}$ be two Morse knots that are isotopic within the space of Morse knots. Recall from Section~\ref{sec:techKI} that $\uZ(K)$ denotes the pre-Kontsevich integral of $K$ in $\D^0/\left\lbrace 1\text{T relations}\right\rbrace$, so that $\uZ(K)-\uZ(\tilde{K})$ is zero modulo $4$T, and that $Z(K)=Z(\tilde{K})\in \A^0$ denotes the class of $\uZ(K)$ modulo $4$T. In other words, $Z$ is the Kontsevich integral for Morse knots, \textit{without the corrective term $Z(\infty)^{-c/2}$}.

\begin{lemma}\label{Lem:prefoncto}{\color{white}.}
\begin{itemize}
\item[(a)] If $\mu_1$ and $\mu_2$ are paths in the space of Morse knots such that the composition $\mu_1\cdot \mu_2$ makes sense, then $$Z^1(\mu_1\cdot \mu_2)= Z^1(\mu_1)+Z^1(\mu_2)$$

In particular, for any path $\mu$, $$Z^1(\mu^{-1})= -Z^1(\mu)$$
\item[(b)] If $K$ and $\tilde{K}$ are isotopic Morse knots as above, and if $\gamma$ is a loop in the space of Morse knots, then
$$(\uZ(K)-\uZ(\tilde{K})) Z^1(\gamma)=0\in\A^1$$
so that the product $Z(K)Z^1(\gamma)$ is well-defined in $\A^1$. Similarly  the product $Z^1(\gamma)Z(K)$ is well-defined in $\A^1$.

\item[(c)] Given $K_1$, $K_2$ and $\gamma$ as above, one has $$Z^1(K_1\hash\gamma\hash K_2)= Z(K_1)Z^1(\gamma)Z(K_2)$$
\end{itemize}
\end{lemma}

Just as the multiplicativity property of the Kontsevich integral allows one to define the corrected integral $\hat{Z}$ (see Section~\ref{sec:techKI}), Point \textsl{(c)} will allow us to correct $Z^1$ into a $1$-cocycle outside the class of Morse knots in the next section. Point \textsl{(b)} is useful not only for \textsl{(c)} to make sense, but also in the proof that some $4$T relations hold in $\A^1$ (Theorem~\ref{thm:bimod}). 
\begin{proof}
Point \textsl{(a)} follows from the additivity property of integrals.
The Fubini theorem gives a pre-version of \textsl{(c)},
$$Z^1(K_1\hash\mu\hash K_2)= \uZ(K_1)Z^1(\mu)\uZ(K_2)$$ after one notices that because $K_1$ and $K_2$ are steady, no $V$ will contribute non-trivially at their levels. Point \textsl{(b)} and subsequently \textsl{(c)} now follow from the application of $Z^1$ to the trivial loop depicted in Figure~\ref{pic:Lem:prefonct}.\qedhere
\begin{figure}[!ht]
\hspace*{-1cm}\labellist
\small\hair 2pt
\pinlabel $L$ at 30 442
\pinlabel $K$ at 30 353
\pinlabel $L$ at 314 442
\pinlabel $K$ at 314 353
\pinlabel $L$ at 30 144
\pinlabel $\tilde{K}$ at 30 58
\pinlabel $L$ at 314 144
\pinlabel $\tilde{K}$ at 314 58
\pinlabel $\stackrel{\text{\normalsize$K\hash\gamma$}}{\xrightarrow{\hphantom{\text{\normalsize$K\hash\gamma^{-1}$}}}}$ at 172 398
\pinlabel $\stackrel{\text{\normalsize$\tilde{K}\hash\gamma^{-1}$}}{\xleftarrow{\hphantom{\text{\normalsize$K\hash\gamma^{-1}$}}}}$ at 172 103
\pinlabel $\uparrow$ at 30 250
\pinlabel $\downarrow$ at 314 250

\endlabellist
\centering 
\hspace{5pt}
\includegraphics[scale=0.29]{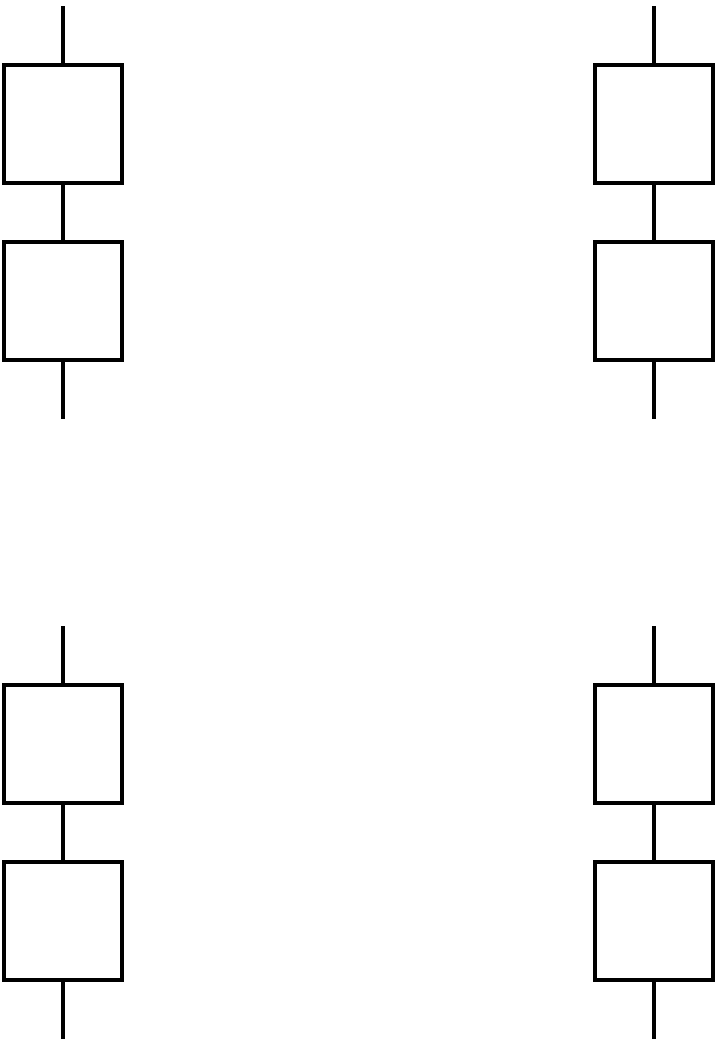}
\caption{$L$ denotes the knot $\gamma(0)$. This loop is trivial within the space of Morse knots, hence $Z^1$ vanishes. The vertical arrows are inverse of each other (pick any path from $K$ to $\tilde{K}$ and connect it with $L$), so their contributions to $Z^1$ cancel out.\qedhere}\label{pic:Lem:prefonct}
\end{figure}

\end{proof}

\section{The corrected integral $\hat{Z}^1$}\label{sec:correctZ1}

We are now ready to define the correction that will get rid of the framing-dependence and make $Z^1$ into a $1$-cocycle in the space of all knots. The correction is essentially the same for our invariant and the Kontsevich integral, which allows us to exhibit a functorial relationship between the two of them in Section~\ref{sec:morefunct}.

\subsection{$\hat{Z}^1$ on paths of Morse knots}
Recall from Section~\ref{sec:techKI} the correction that allows the Kontsevich integral to be a knot invariant outside the class of Morse knots:
$$\hat{Z}(K)=\frac{Z(K)}{Z(\infty)^{c(K)/2}}$$
where $c(K)$ is the number of Morse critical points of $K$ and the symbol $\infty$ stands for a hump (Figure~\ref{pic:hump}). 

Fix a parametrization of the hump, hereafter denoted by $\infty_0$. Its pre-Kontsevich integral $\uZ(\infty_0)$ has an inverse in $\D^0/\left\lbrace 1\text{T relations}\right\rbrace$ as usual via power series. Let $\mu\colon [a,b]\times\R\rightarrow \C\times\R$ be a path in the space of Morse knots. The number of critical points of the knot $\mu_\phi$ is then the same for every value of $\phi$. Call this number $c(\mu)$ and set$$\hat{Z}^1(\mu)= {\uZ(\infty_0)^{-c(\mu)/2}}{Z^1(\mu)}\in \A^1$$

We will see later on (Theorem~\ref{thm:bimod}) that in the case of loops much less care is required, and one can write 
$$\hat{Z}^1(\gamma)= \frac{Z^1(\gamma)}{Z(\infty)^{c(\gamma)/2}}$$

\subsection{$\hat{Z}^1$ on arbitrary paths}

Let $\mu\colon [a,b]\times\R\rightarrow \R^3$ be a path in the space of knots, generic with respect to the Morse function. By this we mean that the knot $\mu(\phi,\cdot)$ is Morse except at finitely many values of $\phi$ in $(a,b)$ where its number of critical points jumps up or down by $2$. We are going to associate to $\mu$ a path $\tilde{\mu}$ in the space of Morse knots. 

Let $K=\mu(a,\cdot)$, and $k$ an integer at least equal to the number of positive perestroikas (birth of a hump) in $\mu$. We set $$\tilde{K}=\infty_0^{\hash k}\hash K,$$
meaning that $K$ receives $k$ factors on the left, all of them isometric to $\infty_0$. The path $\tilde{\mu}$ starts at $\tilde{K}$.
Then, let the path $\mu$ unfold on the right factor, and whenever a perestroika occurs in $\mu$, replace it in $\tilde{\mu}$ by the sliding of a hump either to or from the stock of humps on the left---which can be done while staying within the class of Morse knots (see \cite[Lemma~$3.1$]{MostovoyMorse}).

For a path $\mu$ as above, we set $$\hat{Z}^1(\mu)=\hat{Z}^1(\tilde{\mu}).$$

It does not depend on the choice of $k$ thanks to the property $$Z^1(\infty\hash \mu)=\uZ(\infty)Z^1(\mu)$$ mentioned in the proof of Lemma~\ref{Lem:prefoncto},
and it does not depend on the exact paths followed by the sliding humps by the invariance property of $Z^1$.

By taking limits, one sees that an alternative definition is to set $\hat{Z}^1(\mu)$ to be the sum of $\hat{Z}^1$ on all regular parts of $\mu$ and then add/subtract the cost of moving an infinitesimal hump from $-\infty$ to the place where each perestroika occurs.

\begin{theorem}
$\hat{Z}^1$ is a $1$-cocycle in the space of long knots.
\end{theorem}
\begin{proof}
Given two paths $\mu_1$ and $\mu_2$, and an isotopy $H\colon \mu_1 \rightsquigarrow \mu_2$, after connecting sufficiently large numbers of humps $k_1$ and $k_2$ one can mutate $H$ into an isotopy $\tilde{H}\colon \tilde{\mu}_1 \rightsquigarrow \tilde{\mu}_2$ that stays within the set of Morse knots. The theorem follows then from the invariance of $Z^1$.
\end{proof}

\subsection{More functoriality}\label{sec:morefunct}
\subsubsection*{Connected sum of loops}
The connected sum of two loops $\gamma$ and $\gamma^\prime$ is defined by the loop $\phi\mapsto \gamma^{\phantom{\prime}}_\phi\hash \gamma^\prime_\phi$ after a rescaling to make the time scales match. It is obviously isotopic to the loop $(\gamma \hash K^\prime)\cdot (K \hash \gamma^\prime)$, whereby one can extend the results of Lemma~\ref{Lem:prefoncto}, given that the multiplicative correction from $Z^1$ to $\hat{Z}^1$ and from  $Z$ to $\hat{Z}$ are the same.
\begin{proposition}
Let $\gamma_1$ and $\gamma_2$ be any two loops in the space of knots, respectively in the knot types $K_1$ and $K_2$. Then $$\hat{Z}^1(\gamma_1\hash\gamma_2)=\hat{Z}^1(\gamma_1)\hat{Z}(K_2)+\hat{Z}(K_1)\hat{Z}^1(\gamma_2)$$
\end{proposition}

It was pointed out to us by Victoria Lebed that this makes $\hat{Z}^1$ a Hochschild $1$-cocycle, where $\A^1$ is seen as a module on the monoid algebra $A$ of loops under connected sum, acting via $\gamma\cdot x=\hat{Z}(K)x$, where $x\in \A^1$ and $\gamma$ is a loop in the knot type $K$.

\subsubsection*{The shadow $4$T relations}

Lemma~\ref{Lem:prefoncto} allows us to multiply $Z^1(\gamma)$ and $Z(K)$ within $\A^1$, where $\gamma$ is a Morse loop and $K$ a Morse knot. This a priori says nothing about other common procedures that require $4$T relations, such as division by $Z(K)$, commutativity of the product in $\A^0$, generalized $1$T relations, etc. We shall fix this issue by the following theorem, whose proof makes up the rest of this subsection.

\begin{theorem}\label{thm:bimod}
$\hat{Z}^1(H_1(\mathcal{K}\setminus \Sigma))\subset \A^1$ is an $\A^0$-bimodule. Moreover, for every $\gamma\in H_1(\mathcal{K}\setminus \Sigma)$ and every knot $K$,
$$\hat{Z}^1(\gamma)\hat{Z}(K)=\hat{Z}(K)\hat{Z}^1(\gamma)$$
\end{theorem}

\begin{proof} Pick an integer $m>0$ and let $I_m$ denote the subspace of $\D^0/\left\lbrace 1\text{T}\right\rbrace$ spanned by all differences $\uZ_m(K)-\uZ_m(\tilde{K})$, where $K$ and $\tilde{K}$ are isotopic Morse knots and $\uZ_m$ stands for the part of $\uZ$ of degree exactly $m$. On the other hand, let $J_m$ denote the subspace of $\D^0/\left\lbrace 1\text{T}\right\rbrace$ spanned by all $4$T relators of degree $m$. The invariance of the Kontsevich integral for Morse knots can be stated as $I_m\subseteq J_m$.

We claim that for every $m>0$, 
$$I_m = J_m$$

\textit{Proof of the claim. }Considering a \enquote{strict} Kontsevich integral whose every $m$-th degree part lives in $\D_m^0/\left\lbrace 1\text{T}, I_m\right\rbrace$ rather than ${\D_m}^{\!\!\!\!0}/\left\lbrace 1\text{T}, J_m\right\rbrace$, one can repeat entirely the proof of \cite[Theorem $1$]{BarNatanVKI}, so that every \enquote{strict} weight system $w\colon  \D_m^0/\left\lbrace 1\text{T}, I_m\right\rbrace \rightarrow \C$ is the pull-back of a weight system in the usual sense by the projection $\D_m^0/\left\lbrace 1\text{T}, I_m\right\rbrace \twoheadrightarrow{\D_m}^{\!\!\!\!0}/\left\lbrace 1\text{T}, J_m\right\rbrace$. Hence this projection is an isomorphism. \qed
 
Together with Lemma~\ref{Lem:prefoncto}(b), this implies that if $\delta\in \D^0$ represents the trivial class in $\A^0$, then $\delta Z^1(\gamma)$ and $Z^1(\gamma)\delta$  both represent the trivial class in $\A^1$, where $\gamma$ is a loop in the space of Morse knots. Now since $\hat{Z}^1$ is essentially defined via such loops, we have the first part of the theorem.

For the second part, note that if a loop is connected on the right to a steady factor, this factor can be brought to the left before the loop starts, and brought back to the right afterwards. The resulting loop is isotopic to the original one.
\end{proof}

\section{Relation to Vassiliev $1$-cocycles}\label{sec:Vass}

Similarly to Vassiliev knot invariants, Vassiliev $1$-cocycles are $1$-cocycles whose Alexander duals have a finite depth in the resolved discriminant of the space of knots, and the degree corresponds to the maximal depth required to write down the dual (see Section~\ref{sec:VC}). Again, the chain of maximal depth is represented by a linear combination of diagrams called the principal part of the cocycle---see examples in Fig.~\ref{pic:TTfixed}. Unlike knot invariants, a Vassiliev $1$-cocycle does not have a unique or preferred principal part. However, a principal part still entirely determines a cocycle up to cocycles of lower degree---see Kontsevich's realization theorem, discussed in \cite[Section 4.2.1]{VassilievTT2}. Actual computations to navigate through the resolved discriminant can be found in \cite[Section 3]{VassilievCalc} and \cite[Section 3.3]{FT1cocycles}, respectively to derive a combinatorial formula for a $1$-cocycle given its principal part, and to derive a principal part (and hence show that a cocycle is of finite complexity) given a combinatorial formula.

\subsection{Weight systems of order $1$}\label{subsec:weight}
%
The origin of the $4$T relations lies in Birman--Lin's \cite{BirmanLin} clever rewriting of Vassiliev's equations describing the first level of his spectral sequence \cite{VassilievBook}, whose kernel is made of principal parts of Vassiliev invariants---see Section~\ref{sec:VC}. They show that only one kind of diagram carries essential information in a principal part, namely ordinary chord diagrams, and that after an innocuous change of variables, the equations that rule Vassiliev invariants take a very simple form freed from any local parameter. Combinations of chord diagrams that are solutions to these equations are known today as weight systems, and span the dual space to the target of the Kontsevich integral. The situation is exactly similar here.

We define a \textit{weight system of order $1$ (and degree $m$)} as a linear combination of $V$-diagrams of degree $m$ which, regarded as a functional $w\colon \D^1\rightarrow \C$ via the Kronecker pairing of diagrams, descends to a functional on $\A^1$. Two examples are presented in Figure~\ref{pic:TTfixed2}.

Let $D$ be a $V$-diagram. The \textit{sign} of $D$ is defined by $$S(D)=(-1)^{\sum_{\left\lbrace P,P^\prime\right\rbrace}\lk(P,P^\prime)} $$ where the sum runs over all pairs of either two chords of $D$, or one chord and the $V$. We define the involution $\sigma$ on $\D^1$ by $\sigma(D)=S(D)D$.

\begin{remark}
Topologically, this change of variable amounts to making the opposite choice of orientation for some of the cells in Vassiliev's resolved discriminant. The original orientations of the cells and incidence signs can be found in~\cite[Section $3.3$]{Vassiliev1990}. This is a natural generalization of the sign $S$ defined by Birman--Lin in \cite[p.241]{BirmanLin}, except that their version has an ingredient depending on the degree of $D$, which provides consistency across Vassiliev's \enquote	{actuality tables}. It would be harmless for our purposes to add this ingredient here, but we would gain nothing as we don't have actuality tables for $1$-cocycles---yet?
\end{remark}

\begin{theorem}\label{thm:Vass}
Let $\alpha$ be a Vassiliev $1$-cocycle of degree $m$. Then, after the change of variable $\sigma$, the projection of a principal part of $\alpha$ onto $\D^{1}$ is a weight system of order $1$ and degree $m$. 
\end{theorem}


In other words, there is a natural (multivalued) map $$\V^1_m/\V^1_{m-1}\to \W^1_m$$
where $\V^1_{m}$ denotes the set of Vassiliev $1$-cocycles of degree $m$ and $\W^1_{m}$ the set of weight systems of order $1$ and degree $m$.  
We conjecture that $\hat{Z}^1$ provides an inverse map in the following sense.
\begin{conjecture}\label{conj}
For every $\C$-valued $1$-cocycle $\alpha$ of degree $m$ and corresponding weight system $w$, the $1$-cocycle $\left<w,\hat{Z}^1(\cdot)\right>$ is of finite complexity and differs from $\alpha$ by a $1$-cocycle of degree at most $m-1$.
\end{conjecture}

Since Vassiliev's $1$-cocycles and $\hat{Z}^1$ are ruled by the same algebraic structure, it is natural to expect that a given weight system will define the same cohomology class in both settings. The result holds and is well-known in the case of knot invariants, see \cite[Theorem $1$]{BarNatanVKI}. The rest of this subsection is devoted to the proof of Theorem~\ref{thm:Vass}.


\subsubsection*{Settings of the spectral sequence}
The linear map from Vassiliev's spectral sequence whose kernel consists of principal parts of $1$-cocycles of a given degree $m$ can be described as $$\Psi\colon \Dmt{2}\, \oplus \Dm{1}\oplus \Dm{0,\star}\rightarrow \Dm{2}\oplus \Dm{1,\star}\oplus \Dm{0,\star\star}$$
where 
\begin{itemize}
\item $\Dmt{2}$ is the vector space freely generated by the collection of diagrams that one can obtain by enhancing a $V^2$-diagram of degree $m$ with a chord between two points already indirectly (bigons are not allowed) linked by a $V$ or a $3$-edge tree.
\item $\Dm{0,\star}$ is generated by $m$-chord diagrams enhanced by a lonely star.
\item $\Dm{2}$ is the vector space freely generated by $V^2$-diagrams of degree $m$.
\item $\Dm{1,\star}$ is generated by two kinds of diagrams: $m$-chord diagrams with a star attached to the endpoint of a chord, and $V^1$-diagrams of degree $m$ with a lonely star in $\R$ (away from the chords).
\item $\Dm{0,\star\star}$ is generated by $m$-chord diagrams enhanced by two lonely stars.
\end{itemize}

Up to incidence signs defined in \cite[Section $3.3$]{Vassiliev1990} and simplified in a short note by the author \cite{MortierSimplification}, $\Psi$ maps a generator of $\Dmt{2}$ to the sum of all possible ways to remove a chord whose endpoints remain indirectly connected, and a generator of $\Dm{1}$ or $\Dm{0,\star}$ to the sum of all ways to shrink an admissible interval of $\R$ to a point, which becomes a star when the interval was initially bounded by the two endpoints of an isolated chord. \textit{Admissible }here means that the interval cannot contain the endpoint of a chord in its interior, and has to be bounded by either a star and the endpoint of a chord, or two endpoints of chords, in which case these cannot be the two tips of the $V$.

It is not difficult to see that modulo the image of the preceding map in the spectral sequence, any element in $\Ker\Psi$ has a representative that does not involve diagrams in $\Dm{0,\star}$. Hence we can consider the restriction $$\Psi\colon \Dmt{2}\, \oplus \Dm{1}\rightarrow \Dm{2}\oplus \Dm{1,\star}$$

\subsubsection*{$1$T and $2$T relations}
Considering the preimage of both kinds of generators of $\Dm{1,\star}$ shows respectively that the part in $\Dm{1}$ of any element of $\Ker\Psi$ has to satisfy $1$T and $2$T relations---note that these are unaffected by the change of variable $\sigma$. Assuming these relations, we are now left with a restriction
$$\Psi\colon \Dmt{2}\, \oplus \Dm{1}\rightarrow \Dm{2}$$

\subsubsection*{$16$T and $28$T relations}
We show here how the $16$T and $28$T relations arise from Vassiliev's linear maps. The key parts of the matrices are displayed in Appendix~\ref{Appen:3}.
To understand the equations coming from the generators of $\Dm{2}$ with a $3$-edge tree, we restrict our attention to $16$ such  $V^2$-diagrams that differ only by the way their trees' vertices are connected. The corresponding submatrix of $\Psi$ has: 
\begin{itemize}
\item $16$ rows, one for each diagram from Figure~\ref{pic:spantrees};
\item $15$ columns (say, on the left) corresponding to generators of $\Dmt{2}$;
\item $72$ columns (on the right) corresponding to $V$-diagrams.
\end{itemize}

Denote this $16\!\times\!87$ matrix by $M$ and its $16\!\times\!15$ left submatrix by $M_1$. First we observe that $M_1$ has rank $10$. This means that there are six independent ways to combine the rows of $M$ so as to end up with $15$ zeroes on the left. Denote by $M_2$ the $6\!\times \!72$ matrix on the right of these zeroes: it is the list of all $6$ equations that must be satisfied by the $\Dm{1}$-part of any element of $\Ker M$. It is now a general fact that we have a decomposition
$$\Ker M=\Ker M_1\oplus  E$$
where $E$ is a subspace of $\Ker M$ that is mapped isomorphically onto $\Ker M_2$ by the second projection $\Dmt{2}\, \oplus \Dm{1}\rightarrow \Dm{1}$. In other words, any solution to the six equations in $M_2$ will extend to a solution of the equations in $M$.

One easily checks that $\Ker M_1$ is generated by boundaries from the preceding map in the spectral sequence, so that we are left with the six equations from $M_2$. After the change of variable $\sigma$, they are exactly the three $16$T relations and the three $28$T relations.

\begin{remark}
This process of getting rid of non-essential variables in Vassiliev's spaces by row combinations was already used by Birman--Lin in \cite[Lemma $3.4$]{BirmanLin} to obtain the $4$T relations \cite[Equation $3.10$]{BirmanLin} for the first time.
\end{remark}

\subsubsection*{$\4$T relations}
Finally, to understand the equations coming from those generators of $\Dm{2}$ that have two $V$'s, we enhance an ordinary chord diagram with two full triangles and consider all $9$ \mbox{$V^2$-diagrams} obtained by removing one chord from each triangle. The corresponding submatrix of $\Psi$ has $9$ rows, $6$ $\Dmt{2}$-columns and $36$ $\Dm{1}$-columns. The previous arguments can be repeated and this time we obtain four equations in the end, which are exactly the $\4$T relations described in Subsection~\ref{sec:4x4}. Again the key $9\times 6$ matrix is displayed in Appendix~\ref{Appen:3}.

\subsection{Integration over the Gramain cycle}\label{sec:rot}

The Gramain cycle, denoted by $\rot(K)$, consists of rotating a long knot $K$ once around its axis. It is a loop consistently defined across all path-components of the space of knots, and therefore evaluating a $1$-cocycle on these loops yields an honest knot invariant. The branches of the knot can be parametrized by $$z_i(\phi, t)=z_i(0,t)e^{\sqrt{-1}\phi}\qquad \phi\in [0,2\pi]$$
so that the differential form $\omega_{ij}$ becomes $$\frac{\partial \log(z_i-z_j)|_{\phi=0}}{\partial t}dt + \sqrt{-1}d\phi $$

Integrating then with respect to $\phi$ results in pieces of the Kontsevich integral. Put together, these pieces form a Vassiliev invariant of degree one less than that of $w$, whose weight system can be derived directly from $w$. Namely, let $D$ be a $V$-diagram with $V$ given by $\{\{a,b\},\{b,c\}\}$ with $a<c$. If the interval $(-\infty, a)$ (respectively, $(c, \infty)$) contains no endpoints of any chords in $D$, define $\rho_{-}(D)$ (respectively, $\rho_{+}(D)$) to be the ordinary chord diagram obtained by removing the chord $\{a,b\}$ (respectively, $\{b,c\}$) from $D$. Otherwise set $\rho_{-}(D)$ (respectively, $\rho_{+}(D)$) to $0$. We set:
$$\rho:\D^1\to \D^0, \qquad D\mapsto\frac{\rho_{-}(D)-\rho_{+}(D)}{2}.$$

One can prove by the above argument that for any weight system of order $1$, say $w$, $\rho(w)$ is a weight system of order $0$ and one has $$w(\hat{Z}^1(\rot(K)))=\rho(w)(\hat{Z}(K))$$ 

For example, there are two principal parts of the Teiblum--Turchin cocycle in the literature, \cite[Formula ($10$)]{VassilievCalc} and \cite[Figure $4$]{VassilievTT2}  which we reproduce here in Fig.~\ref{pic:TTfixed} (with a typo fixed in the first one).
Applying Theorem~\ref{thm:Vass} to them yields weight systems $w_1$ and $w_2$ (see Fig.~\ref{pic:TTfixed2}), both of which evaluate on  $\hat{Z}^1(\rot(K))$ as the coefficient of \raisebox{-.15cm}{\includegraphics[scale= 0.3]{v2.pdf}} in $\hat{Z}(K)$, that is the Casson invariant $v_2(K)$.
\begin{figure}[!ht]
\centering
\labellist
\small\hair 2pt
\pinlabel $+$ at 145 63
\pinlabel $+$ at 315 63
\pinlabel $-$ at 485 63
\pinlabel $-$ at 655 63

\endlabellist
\hspace{-3pt}
\includegraphics[scale=0.318]{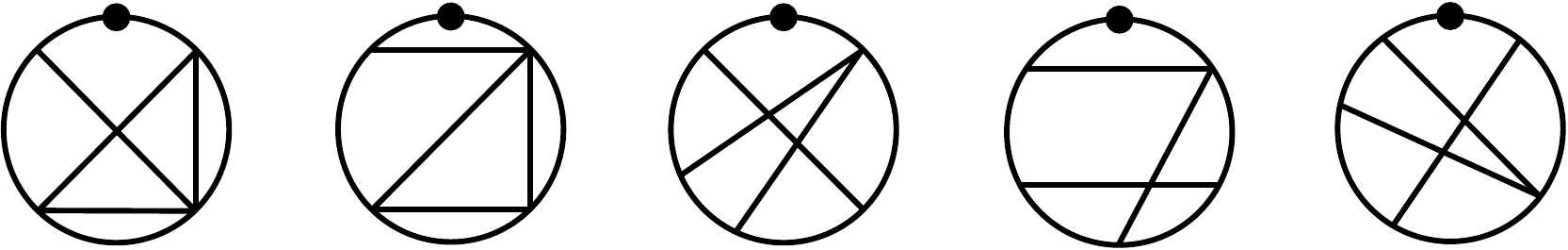}
\labellist
\small\hair 2pt
\pinlabel $-2$ at 140 63
\pinlabel $-$ at 315 63
\pinlabel $+$ at 485 63
\pinlabel $-$ at 145 203
\pinlabel $+$ at 315 203
\pinlabel $-$ at 485 203
\pinlabel $+$ at 655 203
\endlabellist
\hspace*{-2pt}
\includegraphics[scale=0.318]{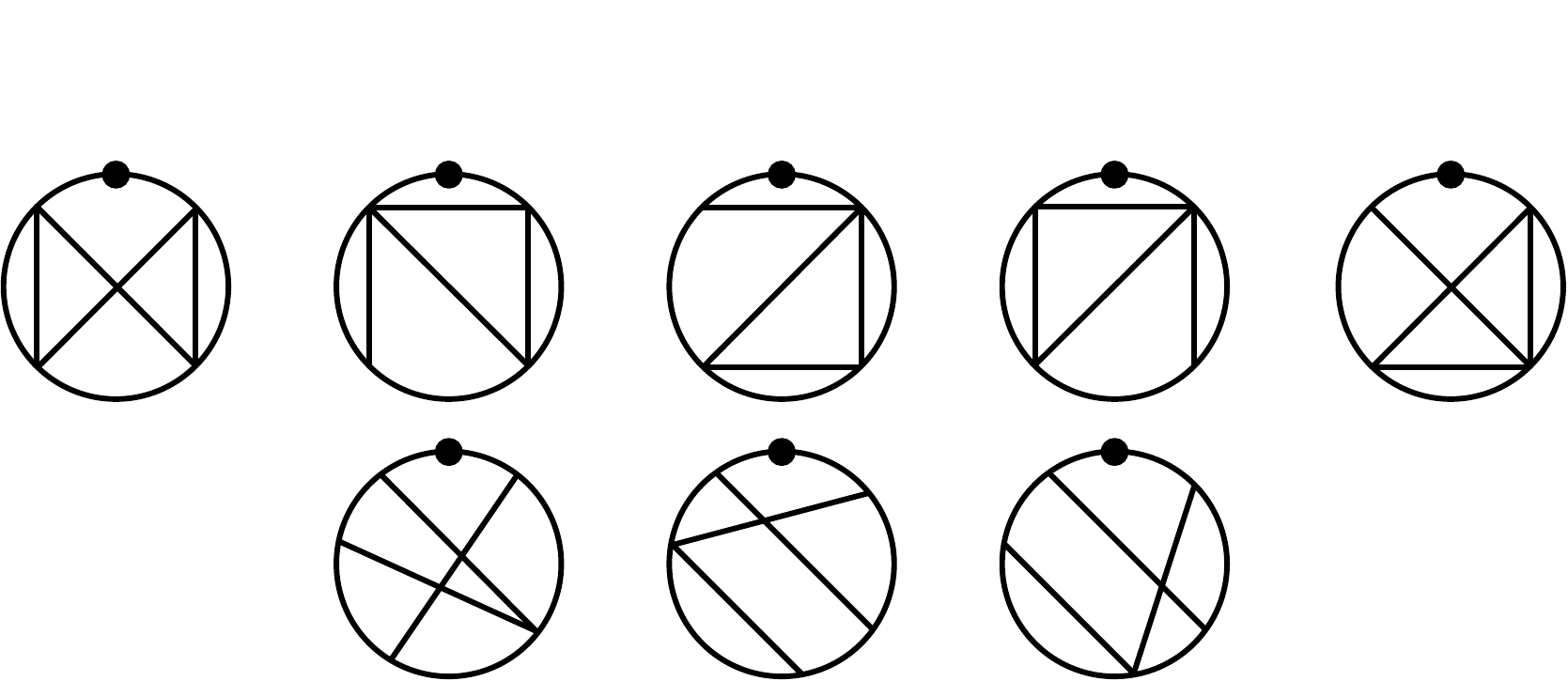}
\caption{Two principal parts of the Teiblum--Turchin $1$-cocycle, due to Vassiliev.\protect\footnotemark}
\label{pic:TTfixed}
\end{figure}

The appearance of the Casson invariant here is exactly what one would expect and brings further support to Conjecture~\ref{conj}. 
Indeed, if the conjecture holds then $w_1(\hat{Z}^1(\cdot))$ is the Teiblum--Turchin cocycle, and there are many reasons to believe that this cocycle evaluates on $\rot(K)$ into the Casson invariant: it was conjectured (and proved over $\Z_2$) by Turchin in \cite{Turchin}, and further support was given by the author in \cite{FT1cocycles, MortierCADS}.

\begin{figure}[!ht]
\centering
\labellist
\small\hair 2pt
\pinlabel $+$ at 145 63
\pinlabel $+$ at 315 63
\pinlabel $w_1=$ at -50 63
\endlabellist
\hspace{-2pt}
\includegraphics[scale=0.318]{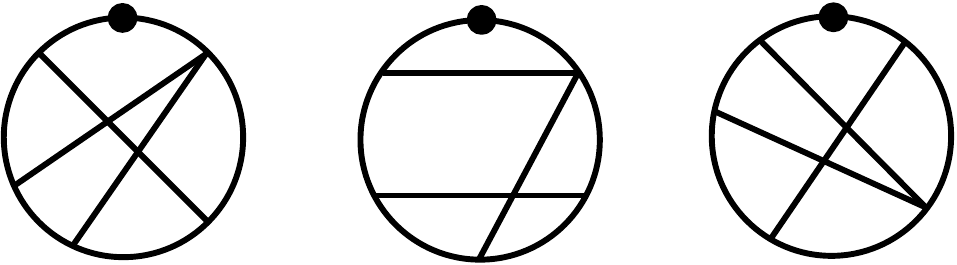}
\labellist
\small\hair 2pt
\pinlabel $w_2=2$ at -59 63
\pinlabel $-$ at 145 63
\pinlabel $+$ at 315 63
\endlabellist
\hspace*{3.3cm}
\includegraphics[scale=0.318]{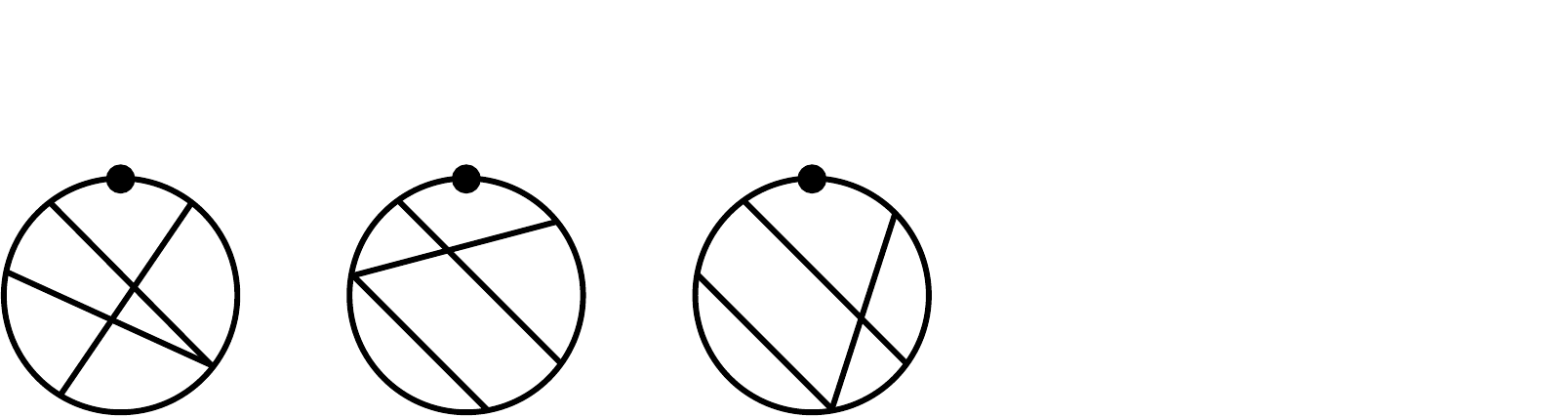}
\caption{Two weight systems of order $1$ and degree $3$ derived from the Teiblum--Turchin $1$-cocycle (Figure~\ref{pic:TTfixed}) as predicted by Theorem~\ref{thm:Vass}. One can see the $2$T relation at work in the last two diagrams of $w_2$.}
\label{pic:TTfixed2}
\end{figure}
\footnotetext{Vassiliev has confirmed in a personal communication that there was a typo in the earlier mentions of the first linear combination, where the second to last sign appeared as a $+$.}

\section{Open problems}\label{sec:open}
\subsection{How strong is $\hat{Z}^1$ and how to find weight systems of order $1$?}
To this day, only one example of a Vassiliev $1$-cocycle is known, namely the Teiblum--Turchin cocycle (see Section~\ref{sec:rot}). The present results could lead to systematic methods to find more. A brute-force attack would be possible, as the $1$T to $28$T relations are perhaps easier to encode than the original Vassiliev equations, however the size of the system of equations still grows extremely fast with the degree of the invariants.
Here are some possible ideas to try to circumvent this.

\begin{question}\label{quest:1}
Is there a non-trivial map $\theta:\D^1\to \D^0$ such that $$\theta(\Span\left\{1\text{T}, 2\text{T}, \4\text{T}, 16\text{T}, 28\text{T} \right\})\subseteq \Span\left\{1\text{T}, 4\text{T} \right\}?$$
\end{question}
If so, then the dual map $\theta^\star$ will send any weight system of order $0$ (of which there are many, see e.g. Bar-Natan's account in  \cite[Section 1.4]{BarNatanVKI}) to a weight system of order $1$. Naive attempts in this direction so far have only led to homologically trivial weight systems---that is, lying in the image of the map $\partial_\nabla$ from Section~\ref{sec:4Tin}.

One could make this question even more ambitious:
\begin{question}
Is there a map $\theta$ as in Question~\ref{quest:1} whose dual is a right inverse to the integration map $\rho:\D^1\to \D^0$ from Section~\ref{sec:rot}?
\end{question}
This would imply that not only there are infinitely many weight systems of order $1$, but $\hat{Z}^1(\rot (K))$ would already contain as much information about $K$ as $\hat{Z}(K)$ does, since every classical Vassiliev invariant would arise as the evaluation of some weight system of order $1$ on the Gramain cycle.

Finally, one remarkable way to produce weight systems of degree $0$ was found by Bar-Natan, taking as input Lie algebraic data, and where the $4$T relation translates to the Jacobi identity---see \cite[Section $2.4$ and Theorem $4$]{BarNatanVKI} and \cite{BN1995weights}.

\begin{problem}
Is there a way to derive weight systems from Lie algebras or possibly higher structures? If the $4\!$T relations relate to the Jacobi identity, what do the $16/28$T relations relate to?
\end{problem}

\subsection{Further structural properties of $\hat{Z}^1$.}
At this point we have exhibited strong evidence that our integral $\hat{Z}^1$ is related with Vassiliev $1$-cocycles, and Theorem~\ref{thm:Vass} can be seen as one half of a generalization of \cite[Theorem 1]{BarNatanVKI}. The other half would be

\begin{problem}[Conjecture~\ref{conj}]
Prove that a weight system evaluated on $\hat{Z}^1$ outputs a \emph{Vassiliev} $1$-cocycle, of the same degree, whose principal part is essentially the initial weight system.
\end{problem}


As a teaser to our last problem, let us mention a beautiful formula suggested to us by Faria Martins, who has foreseen from a categorical point of view that this simply ought to be true---and it is indeed. Let $\mu$ be a path in the space of Morse knots, from $K_0$ to $K_1$. Then
$$\partial Z^1(\mu)=\uZ(K_1)-\uZ(K_0)$$

where $\partial$ is the dual of Vassiliev's boundary map $\partial_V$ mentioned in Section~\ref{sec:4Tin}, up to the incidence signs. We already knew from the invariance of the Kontsevich integral that $\uZ(K_1)-\uZ(K_0)$ lies in the subspace spanned by the $4$T relators. It turns out that it is not just any combination of $4$T relators, it is the shadow of a higher invariant that keeps track of how exactly $K_0$ goes to $K_1$.

\begin{problem}
How far can one develop the functoriality properties of $Z$ and $Z^1$, in particular the categorical aspects related to Cirio--Faria Martins' \cite{CirioFaria}?
\end{problem}

%
%

%

\appendix
\section{Appendix: Some details in the proof of cocyclicity of $Z^1$}\label{Appen:1}

The fact that $Z^1$ vanishes on a stratum of type $t_j=t_{j+1}=t_{j+2}$ is proved by the identity $\Omega_n\wedge\Lambda_n-\Lambda_n\wedge\Omega_n=0$. Indeed, this stratum corresponds to an ordinary chord reaching the level of the $V$, which can occur from two directions with opposite incidence signs.

Now in the expansion of $\Omega_n\wedge\Lambda_n-\Lambda_n\wedge\Omega_n$, one can immediately discard the contributions where five strands are involved (because a ($\C$-valued) $2$-form will commute with a $1$-form and the chord diagrams also commute in this case), as well as those with only three strands involved (because the exterior product of the corresponding forms vanishes already).

One is left with the four-strand contributions, with diagrams that are desingularizations of spanning trees of a complete $4$-vertex graph. The following lemma is proved in \cite[Lemma $5.1$]{MortierSimplification}.

\begin{lemma}
Let $T$ be a tree with $p$ vertices, labelled from $1$ to $p$. The following differential form on $\C^p$, defined up to sign,
$$\omega_T=\bigwedge_{\text{all edges $\left\lbrace i, j\right\rbrace$ of }T} dz_i-dz_j$$
is equal (up to sign) to the form $$\omega_p=\sum_{i=1}^p (-1)^i dz_1\ldots \widehat{dz_i}\ldots dz_p$$
\end{lemma}

It follows that, in every contribution involving four strands, say $z_1$, $z_2$, $z_3$, $z_4$, one can set aside as an overall factor the form $$\frac{\omega_4}{(z_1-z_2)(z_1-z_3)(z_1-z_4)(z_2-z_3)(z_2-z_4)(z_3-z_4)}$$
and every summand in what remains is the product of a polynomial of degree $3$ in the variables $z_i$ with some $V$-diagram on $n$-strands.

There are $20$ monomials of degree $3$ in four variables, but those of the form $z_i^3$ never contribute, since it would mean that $z_i$ is not involved in some denominator, a contradiction with $T$ being a tree.

So the identity $\Omega_n\wedge\Lambda_n-\Lambda_n\wedge\Omega_n=0$ is equivalent to the vanishing of $16$ combinations of $V$-diagrams. The corresponding $16\!\times\!72$ matrix has rank $6$ and the same kernel as the matrix $M_2$ from Subsection~\ref{subsec:weight}---so both sets of rows span the same space, which means that $\Omega_n\wedge\Lambda_n-\Lambda_n\wedge\Omega_n=0$ is equivalent to the $16$T and $28$T relations.

The study of strata \enquote{$t_j=t_{j+1} < t_k=t_{k+1}$} is similar using the $\4$T relations.

\begin{remark}
Using the Arnold identity \eqref{Arnold}, the $2$-form $\Lambda_n$ can be rewritten up to a constant factor as the following, summed over all $1\leq i <j<k\leq  n$: $$ \left(\, \protect\lmbn\,\right)\omega_{ij}\wedge\omega_{jk} \,\,\,+\,\,\, \left(\, \protect\lmbnn\,\right)\omega_{ij}\wedge\omega_{ik}$$

Hence, the theory of Cirio--Faria Martins~\cite{CirioFaria} can be applied: using the left action $\smalltriangleright$ of chord diagrams on $n$ strands on $V$-diagrams on $n$-strands given by $a\smalltriangleright b=ab-ba$ and the couple $(A,B)=(\Omega_n, \Lambda_n)$, the $2$-curvature of $(A,B)$ is exactly $\Omega_n\wedge\Lambda_n-\Lambda_n\wedge\Omega_n$, while the $16$T and $28$T relations are equivalent to the six relations from \cite[Theorem $10$]{CirioFaria}. The first result discussed in this appendix can therefore be regarded as a particular case of this theorem.
\end{remark}

\section{Appendix: Key matrices in Vassiliev's spectral sequence}\label{Appen:3}

\setcounter{MaxMatrixCols}{30}
We give here the left submatrices from Subsection~\ref{subsec:weight} which are the key to find the $16$T, $28$T and $\4$T relations in the Vassiliev settings. The right submatrices are too large to be displayed here but can be computed easily using a simplification of Vassiliev's incidence signs by the author in \cite[Theorem $4.1$]{MortierSimplification}. Here is the $9\!\times\! 6$ matrix from Paragraph $\4$T relations.
\[\begin{array}{c||c|c|c||c|c|c|}
&\includegraphics[ scale=0.2]{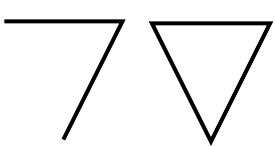}&\includegraphics[ scale=0.2]{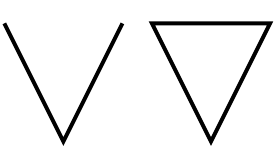}&\includegraphics[ scale=0.2]{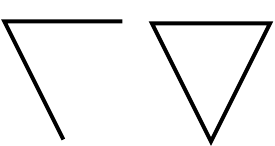}&\includegraphics[ scale=0.2]{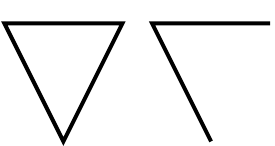}&\includegraphics[ scale=0.2]{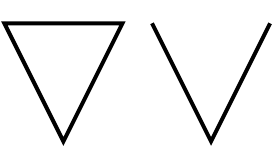}&\includegraphics[ scale=0.2]{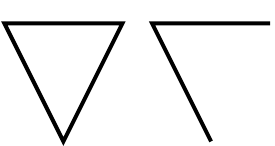}\\\hline \hline 
\raisebox{-0.1cm}{\includegraphics[ scale=0.2]{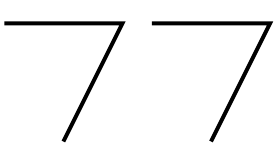}}&+&&&+&&\\\hline 
\raisebox{-0.1cm}{\includegraphics[ scale=0.2]{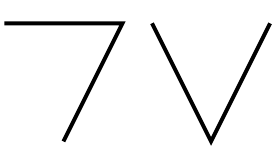}}&-&&&&+&\\\hline
\raisebox{-0.1cm}{\includegraphics[ scale=0.2]{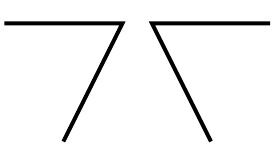}}&+&&&&&+\\\hline\hline
\raisebox{-0.1cm}{\includegraphics[ scale=0.2]{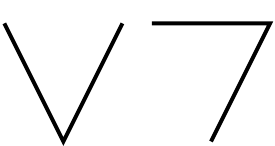}}&&+&&-&&\\\hline 
\raisebox{-0.1cm}{\includegraphics[ scale=0.2]{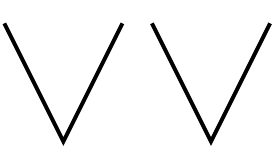}}&&-&&&-&\\\hline
\raisebox{-0.1cm}{\includegraphics[ scale=0.2]{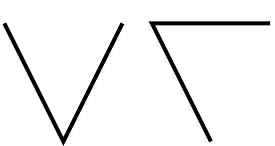}}&&+&&&&-\\\hline\hline
\raisebox{-0.1cm}{\includegraphics[ scale=0.2]{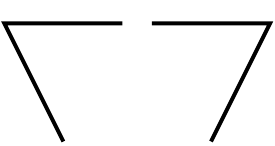}}&&&+&+&&\\\hline
\raisebox{-0.1cm}{\includegraphics[ scale=0.2]{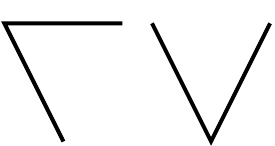}}&&&-&&+&\\\hline 
\raisebox{-0.1cm}{\includegraphics[ scale=0.2]{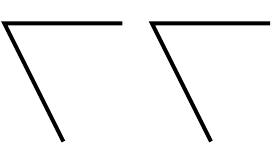}}&&&+&&&+\\\hline
\end{array}\]

Its kernel is $1$-dimensional generated by the boundary of the diagram \raisebox{-0.1em}{\includegraphics[ scale=0.2]{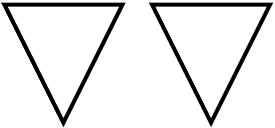}}, so it does not contribute to the homology. The kernel of its transpose, however, is generated by the following, which yield the $\4$T relations.
\[\begin{matrix}
(&1&1&\cdot&&1&1&\cdot&&\cdot&\cdot&\cdot&)\\(&\cdot&1&1&&\cdot&1&1&&\cdot&\cdot&\cdot&)\\(&\cdot&\cdot&\cdot&&1&1&\cdot&&1&1&\cdot&)\\(&\cdot&\cdot&\cdot&&\cdot&1&1&&\cdot&1&1&)
\end{matrix}\]

Now below is the $16\!\times\!15$ matrix $M_1$ from Subsection~\ref{subsec:weight}, Paragraph $16$T and $28$T relations. The incidence signs are $(-1)^{j-1}$ where $j$ is the label of the removed chord, with the chords labeled from $1$ to $4$ lexicographically according to the ordering of the vertices: {\raisebox{-1em}{\labellist
\small\hair 2pt
\pinlabel $1$ at 5 90
\pinlabel $2$ at 5 5
\pinlabel $3$ at 95 5
\pinlabel $4$ at 95 90
\endlabellist \includegraphics[height=2.5em]{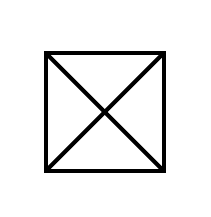}}}

\[\begin{array}{c||c|c|c||c|c|c||c|c|c||c|c|c||c|c|c|}
&\includegraphics[ scale=0.2]{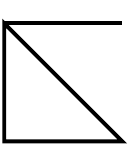}&\includegraphics[ scale=0.2]{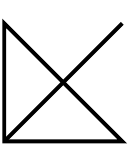}&\includegraphics[ scale=0.2]{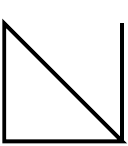}&\includegraphics[ scale=0.2]{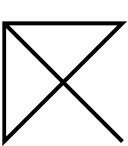}&\includegraphics[ scale=0.2]{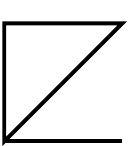}&\includegraphics[ scale=0.2]{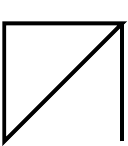}&\includegraphics[ scale=0.2]{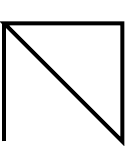}&\includegraphics[ scale=0.2]{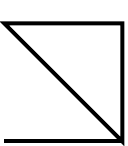}&\includegraphics[ scale=0.2]{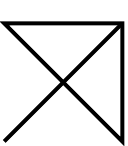}&\includegraphics[ scale=0.2]{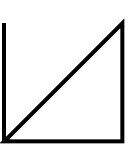}&\includegraphics[ scale=0.2]{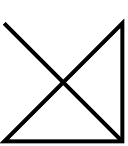}&\includegraphics[ scale=0.2]{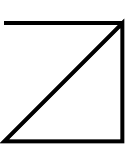}&\includegraphics[ scale=0.2]{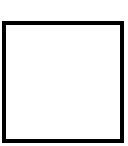}&\includegraphics[ scale=0.2]{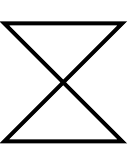}&\includegraphics[ scale=0.2]{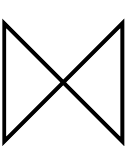}\\\hline \hline 
\raisebox{-0.1cm}{\includegraphics[ scale=0.2]{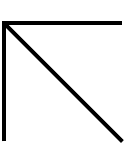}}&-&&&-&&&-&&&&&&&&\\\hline 
\raisebox{-0.1cm}{\includegraphics[ scale=0.2]{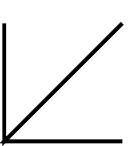}}&&-&&&-&&&&&-&&&&&\\\hline
\raisebox{-0.1cm}{\includegraphics[ scale=0.2]{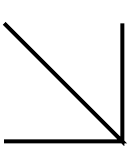}}&&&+&&&&&-&&&+&&&&\\\hline
\raisebox{-0.1cm}{\includegraphics[ scale=0.2]{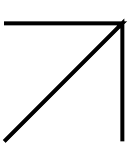}}&&&&&&+&&&+&&&-&&&\\\hline\hline 
\raisebox{-0.1cm}{\includegraphics[ scale=0.2]{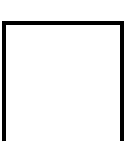}}&&&&&&+&-&&&&&&+&&\\\hline
\raisebox{-0.1cm}{\includegraphics[ scale=0.2]{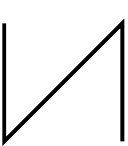}}&&&&&&-&&&&-&&&&&-\\\hline
\raisebox{-0.1cm}{\includegraphics[ scale=0.2]{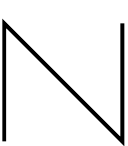}}&&&+&&&&+&&&&&&&&+\\\hline
\raisebox{-0.1cm}{\includegraphics[ scale=0.2]{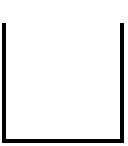}}&&&-&&&&&&&+&&&-&&\\\hline\hline 
\raisebox{-0.1cm}{\includegraphics[ scale=0.2]{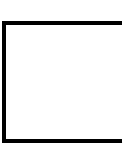}}&-&&&&-&&&&&&&&-&&\\\hline
\raisebox{-0.1cm}{\includegraphics[ scale=0.2]{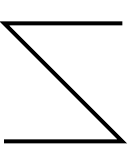}}&+&&&&&&&-&&&&&&-&\\\hline
\raisebox{-0.1cm}{\includegraphics[ scale=0.2]{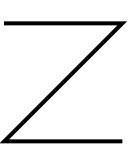}}&&&&&+&&&&&&&-&&+&\\\hline
\raisebox{-0.1cm}{\includegraphics[ scale=0.2]{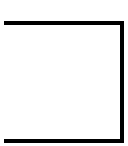}}&&&&&&&&+&&&&+&+&&\\\hline\hline 
\raisebox{-0.1cm}{\includegraphics[ scale=0.2]{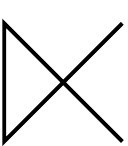}}&&+&&+&&&&&&&&&&&-\\\hline
\raisebox{-0.1cm}{\includegraphics[ scale=0.2]{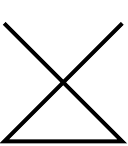}}&&+&&&&&&&&&-&&&-&\\\hline
\raisebox{-0.1cm}{\includegraphics[ scale=0.2]{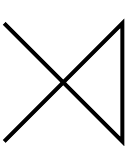}}&&&&&&&&&-&&-&&&&+\\\hline
\raisebox{-0.1cm}{\includegraphics[ scale=0.2]{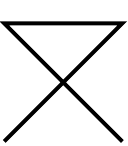}}&&&&+&&&&&-&&&&&+&\\\hline
\end{array}\]

One can see that $\Ker M_1$ is generated by any five of the boundaries of the diagrams \raisebox{-0.1em}{\includegraphics[ scale=0.2]{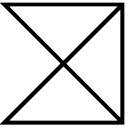}}, \raisebox{-0.1em}{\includegraphics[ scale=0.2]{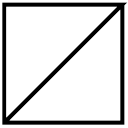}}, \raisebox{-0.1em}{\includegraphics[ scale=0.2]{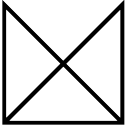}}, \raisebox{-0.1em}{\includegraphics[ scale=0.2]{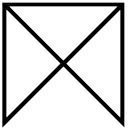}}, \raisebox{-0.1em}{\includegraphics[ scale=0.2]{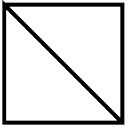}}, \raisebox{-0.1em}{\includegraphics[ scale=0.2]{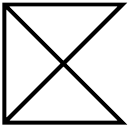}} by the preceding map in the spectral sequence, so it does not contribute to the homology.

On the other hand, $\Ker M_1^\text{T}$ is generated by the six following vectors.
\[\begin{matrix}
(&\cdot&\cdot&\cdot&\cdot&&1&1&1&1&&\cdot&\cdot&\cdot&\cdot&&\cdot&\cdot&\cdot&\cdot&)\\(&\cdot&\cdot&\cdot&\cdot&&\cdot&\cdot&\cdot&\cdot&&1&1&1&1&&\cdot&\cdot&\cdot&\cdot&)\\(&\cdot&\cdot&\cdot&\cdot&&\cdot&\cdot&\cdot&\cdot&&\cdot&\cdot&\cdot&\cdot&&1&-1&1&-1&)\\(&1&1&\cdot&\cdot&&\cdot&\cdot&1&1&&-1&\cdot&\cdot&\cdot&&1&\cdot&\cdot&\cdot&)\\(&\cdot&1&1&\cdot&&\cdot&\cdot&\cdot&1&&\cdot&\cdot&1&1&&\cdot&1&\cdot&\cdot&)\\(&\cdot&\cdot&1&1&&\cdot&1&\cdot&1&&\cdot&\cdot&\cdot&1&&\cdot&\cdot&1&\cdot&)
\end{matrix}\]

After the change of variable $\sigma$, Vassiliev's incidence signs coincide with the signs defining our compact variables in Notation~\ref{not:16trees}, up to Vassiliev's $\varepsilon\zeta$ which is $-1$ exactly in the case of\,  \raisebox{-0.1em}{\includegraphics[scale=0.2]{I}}\,, \raisebox{-0.1em}{\includegraphics[scale=0.2]{J}} and \raisebox{-0.1em}{\includegraphics[scale=0.2]{P}} (see \cite[Theorem $3.1$, Example $4.3$]{MortierSimplification}. After changing the signs accordingly in Columns $8$, $9$ and $16$ in the six vectors above, one recovers the $16$T and $28$T relations as expected.

\bibliographystyle{plain}
\bibliography{bibli}

\end{document}